\documentclass[11pt]{article}

\title{Large time behaivor of global solutions\\ to 
nonlinear wave equations with frictional \\ and viscoelastic damping terms}
\author{
Ryo Ikehata,\\
Department of Mathematics, \\
Graduate School of Education, Hiroshima University\\
Higashi-Hiroshima 739-8524, Japan \\
\ \\
Hiroshi Takeda,\\
Department of Intelligent Mechanical Engineering, \\
Faculty of Engineering, Fukuoka Institute of Technology, \\
3-30-1 Wajiro-higashi, Higashi-ku, Fukuoka, 811-0295 JAPAN 
}
\date{}

\usepackage{amssymb,amsmath,amsthm}
\usepackage[dvips]{graphicx}
\usepackage{bm} 

\newcommand{\R}{\mathbb R}

\newcommand{\K}{\mathcal{K}}

\newcommand{\supp}{\mathop{\mathrm{supp}}\nolimits}

\newtheorem{thm}{Theorem}[section]
\newtheorem{cor}[thm]{Corollary}
\newtheorem{prop}[thm]{Proposition}
\newtheorem{lem}[thm]{Lemma}
\theoremstyle{remark}
\newtheorem{rem}[thm]{Remark}

\theoremstyle{definition}

\begin{document}
\maketitle

\numberwithin{equation}{section}

\begin{abstract}
In this paper,  we study the Cauchy problem for a nonlinear wave equation with 
frictional and viscoelastic damping terms.
As is pointed out by \cite{IS}, 
in this combination, the frictional damping term is dominant for the viscoelastic one for the global dynamics of 
the linear equation. In this note we observe that if the initial data is small, 
the frictional damping term is again dominant even in the nonlinear equation case.
In other words, 
our main result is diffusion phenomena: 
the solution is approximated by the heat kernel with a suitable constant.
Our proof is based on several estimates for the corresponding linear equations.
\end{abstract}


\newpage
\section{Introduction}

In this paper we are concerned with the following Cauchy problem for a wave equation with two types of damping terms 
\begin{equation} \label{eq:1.1}
\left\{
\begin{split}
& \partial_{t}^{2} u -\Delta u + \partial_{t} u -\Delta \partial_{t} u =f(u), \quad t>0, \quad x \in \R^{n}, \\
& u(0,x)=u_{0}(x), \quad \partial_{t} u(0,x)=u_{1}(x) , \quad x \in \R^{n}, 
\end{split}
\right.
\end{equation}
where $u_{0}(x)$ and $u_{1}(x)$ are given initial data, and about the nonlinearity $f(u)$ we shall consider only the typical case such as
\[f(r) := \vert r\vert^{p},\quad (p > 1),\]
without loss of generality. \\

In the Cauchy problem case of the following equation with a frictional damping term
\begin{equation} \label{eq:1.1.0}
\left\{
\begin{split}
& \partial_{t}^{2} u -\Delta u + \partial_{t} u = f(u), \quad t>0, \quad x \in \R^{n}, \\
& u(0,x)=u_{0}(x), \quad \partial_{t} u(0,x)=u_{1}(x) , \quad x \in \R^{n}, 
\end{split}
\right.
\end{equation}
nowadays one knows an important result called as the critical exponent problem such as: there exists an exponent $p_{*} > 1$ such that if the power $p$ of nonlinearity $f(u)$ satisfies $p_{*} < p$, then the corresponding problem \eqref{eq:1.1.0} has a small data global in time solution, while in the case when $1 < p \leq p_{*}$ the problem \eqref{eq:1.1.0} does not admit any nontrivial global solutions. In this frictional damping case, one has $p_{*} = p_{F} := 1 + \displaystyle{\frac{2}{n}}$, which is called as the Fujita exponent in the semi-linear heat equation case. About these contributions, one can cite so many research papers written by \cite{HKN}, \cite{HO}, \cite{IMN}, \cite{ITani}, \cite{K}, \cite{KU}, \cite{MN}, \cite{Na}, \cite{N-2}, \cite{T-2}, \cite{TY}, \cite{Z} and the references therein. \\

Quite recently, Ikehata-Takeda \cite{IT} has treated the 
original problem \eqref{eq:1.1} motivated by a previous 
result concerning the linear equation due to 
Ikehata-Sawada \cite{IS}, and solved the Fujita critical 
exponent one. They have discovered the value $p_{*} = 
1+\displaystyle{\frac{2}{n}}$ again only in the low 
dimensional case (i.e., $n = 1,2$).  So, the problem is still open for $n \geq 3$. Anyway, this result due to \cite{IT} implies an important recognition that the 
dominant term is still the frictional damping $\partial_{t} u$ although the equation \eqref{eq:1.1} has 
two types of damping terms. Note that in the viscoelastic 
damping case:
\begin{equation} \label{eq:1.1.1}
\left\{
\begin{split}
& \partial_{t}^{2} u -\Delta u -\Delta\partial_{t} u = f(u), \quad t>0, \quad x \in \R^{n}, \\
& u(0,x)=u_{0}(x), \quad \partial_{t} u(0,x)=u_{1}(x) , \quad x \in \R^{n}, 
\end{split}
\right.
\end{equation}
we still do not know the ``exact" critical exponent $p_{*}$. Several interesting results about this critical exponent problem including optimal linear estimates for \eqref{eq:1.1.1} can be observed in the literature due to D'Abbicco-Reissig \cite[see Theorem 2, and Section 4]{DR}. But, it seems a little far from complete solution on the critical exponent problem of \eqref{eq:1.1.1}. In fact, in \cite{DR} they studied more general form of equations such as
\[ \partial_{t}^{2} u -\Delta u + (-\Delta)^{\sigma}\partial_{t} u = \mu f(u)\]   
with $\sigma \in [0,1]$ and $\mu \geq 0$. Pioneering and/or important contributions for the case $\sigma = 1$ (i.e., strong damping one) can be seen in some papers due to \cite{Ik-4}, \cite{ITY} ( both in abstract theory), \cite{Ponce}, \cite{Shibata} and the references therein.\\

From observations above one naturally encounters an important problem such as:\\
\noindent
even in the higher dimensional case for $n \geq 3$, can one also solve the critical exponent problem of \eqref{eq:1.1}?\\
Our first purpose is to prove the following global existence result of the solution together with suitable decay properties to problem \eqref{eq:1.1}.
\begin{thm} \label{Thm:1.1}
Let $n=1,2,3$, $\varepsilon>0$ and $p> 1+\displaystyle{\frac{2}{n}}$.
Assume that 
$(u_{0}, u_{1}) \in (W^{\frac{n}{2}+\varepsilon,1} \cap  W^{\frac{n}{2}+\varepsilon,\infty})
 \times (L^{1} \cap L^{\infty})$ with sufficiently small norms. Then, there exists a unique global solution $u \in C([0, \infty);L^{1} \cap L^{\infty})$ to  problem \eqref{eq:1.1}
satisfying
\begin{equation}
\begin{split}
\| u(t,\cdot) \|_{L^{q}(\R^{n})} \le C(1+t)^{-\frac{n}{2}(1-\frac{1}{q})}
\end{split}
\end{equation}
for $q \in [1, \infty]$.
\end{thm}
Our second aim is to study the large time behavior of the global solution given in Theorem {\rm \ref{Thm:1.1}}.
\begin{thm} \label{Thm:1.2}
Under the same assumptions as in Theorem {\rm \ref{Thm:1.1}},  
the corresponding global solution $u(t,x)$ satisfies 
\begin{equation}
\begin{split}
\lim_{t \to \infty} t^{\frac{n}{2}(1-\frac{1}{q}) } 
\| u(t,\cdot)-MG_{t} \|_{L^{q}(\R^{n})} =0,
\end{split}
\end{equation}
for $1 \le q \le \infty$,
where $M:= \displaystyle{\int_{\R^{n}}} (u_{0}(y) +u_{1}(y)) dy + \displaystyle{\int_{0}^{\infty} \int_{\R^{n}}} f(u(s,y)) dy ds$.
\end{thm}

\begin{rem}{\rm By combining the blowup result given in \cite[Theorem 1.3]{IT} and Theorems \ref{Thm:1.1} and \ref{Thm:1.2} with $n = 3$, one can make sure that even in the $n = 3$ dimensional case the critical exponent $p_{*}$ of the nonlinearity $f(u)$ is given by the Fujita number $p_{*} = p_{F}$. Such sharpness has already been announced in the low dimensional cases (i.e., $n = 1,2$) by \cite[Theorems 1.1 and 1.3]{IT}. So, the result for $n = 3$ is essentially new. This is one of our main contributions to problem \eqref{eq:1.1} in this paper. It is still open to show the global existence part for all $n \geq 4$.}
\end{rem}

Before closing this section, 
we summarize notation, which will be used throughout this paper.\\

Let $\hat{f}$ denote the Fourier transform of $f$
defined by
\begin{align*}
\hat{f}(\xi) := c_{n}
\int_{\R^{n}} e^{-i x \cdot \xi} f(x) dx
\end{align*}
with $c_{n}= (2 \pi)^{-\frac{n}{2}}$.
Also, let $\mathcal{F}^{-1}[f]$ or $\check{f}$ denote the inverse
Fourier transform.

We introduce smooth cut-off functions to localize the frequency region as follows:
 
$\chi_L $, $\chi_M$ and $\chi_H  \in C^{\infty}(\mathbb{R})$ are defined by 
\begin{gather*}
\chi_L (\xi) = \begin{cases}
	1, \quad &|\xi| \leq \frac{1}{2}, \\
	0, \quad &|\xi| \geq \frac{3}{4}, 
	\end{cases} \qquad
\chi_H (\xi) = \begin{cases}
	1, \quad &|\xi| \geq 3, \\
	0, \quad &|\xi| \leq 2, 
	\end{cases} \\ 
\chi_M (\xi) = 1- \chi_L (\xi) - \chi_H (\xi). 
\end{gather*}

For $k \ge 0$ and $1 \le p \le \infty$, let $W^{k,p}(\R^{n})$ be the usual Sobolev spaces
\begin{equation*}
W^{k,p}(\R^{n})
  :=\Big\{ f:\R^{n} \to \R;
        \| f \|_{W^{k,p}(\R^{n})} 
        := \| f \|_{L^{p}(\R^{n})}+
         \|  |\nabla_{x}|^{k} f \|_{L^{p}(\R^{n})}< \infty 
     \Big\},
\end{equation*}
where $L^{p}(\R^{n})$ is the Lebesgue space for $1 \le p \le \infty$ as usual.
When $p=2$, we denote $W^{k,2}(\R^{n}) = H^{k}(\R^{n})$.
For the notation of the function spaces, 
the domain $\R^{n}$ is often abbreviated.
We frequently use the notation $\| f \|_{p} =\| f \|_{L^{p}(\R^{n})}$ without confusion.
Furthermore, in the following $C$ denotes a positive constant, which may change from line to line.

The paper is organized as follows. 
Section 2 presents some preliminaries.
In Section 3,
we show the point-wise estimates of the propagators for the corresponding linear equation in the Fourier space. 
Section 4 is devoted to the proof of linear estimates, which play crucial roles to get main results.
In sections 5 and 6, 
we give the proof of our main results.

\section{Preliminaries}
In this section, we collect several basic facts on the Fourier multiplier theory, 
the decay estimates of the solution for the heat equation 
and elementary inequalities to obtain the decay property of the solutions. 

\subsection{Fourier multiplier}
For $f \in L^{2} \cap L^{p}$,
$1 \le p \le \infty$,
let $m(\xi)$ be the Fourier multiplier defined by
\begin{align*}
\mathcal{F}^{-1}[m \hat{f}](x)
= c_{n} \int_{\R^{n}}
e^{-i x \cdot \xi} m(\xi) \hat{f}(\xi) d \xi.
\end{align*}
We define $M_{p}$ as the class of the Fourier multiplier
with $1 \le p \le \infty$:
\begin{align*}
M_{p} := \biggl\{
m:\R^{n} \to \R\,\vert\,
{\rm There\ exists\ a\ constant}\ A_{p}>0\
{\rm such\ that\ } \| \mathcal{F}^{-1}[m \hat{f}] \|_{p}
\le A_{p} \| f \|_{p}
\biggr\}.
\end{align*}
For $m \in M_{p}$,
we let
\begin{align*}
M_{p}(m) := \sup_{f \neq 0}
\frac{\|\mathcal{F}^{-1}[m \hat{f}] \|_{p}}
{\| f \|_{p} }.
\end{align*}

The following lemma describes the inclusion among
the class of multipliers.
\begin{lem}
\label{Lem:2.1}
Let $\displaystyle{\frac{1}{p}}+ \displaystyle{\frac{1}{p'}} = 1$ with
$1 \le p \le p' \le \infty$.
Then $M_{p}=M_{p'}$ and for $m \in C^{\infty}(\R^{n})$, it holds that 
$$
M_{p}(m) =M_{p'}(m).
$$
Moreover, if $m \in M_{p}$,
then $m \in M_{q}$ for all $q \in [p,p']$ and
\begin{align} \label{eq:2.1}
M_{q}(m) \le M_{p}(m)=M_{p'}(m).
\end{align}
\end{lem}

\noindent
We use the Carleson-Beurling inequality,
which is applied to show the $L^{p}$ boundedness
 of the Fourier multipliers.
\begin{lem}[Carleson-Beurling's inequality]
 \label{Lem:2.2}
If $m \in H^{s}$ with $s >\displaystyle{\frac{n}{2}}$,
then $m \in M_{r}$ for all $1 \le r \le \infty.$
Moreover,
there exists a constant $C>0$ such that
	\begin{align} \label{eq:2.2}
		M_{\infty} (m)
		\le C
\| m \|^{1-\frac{n}{2s}}_{2}
\| m \|^{\frac{n}{2s}}_{\dot{H}^{s}}.
	\end{align}
\end{lem}
For the proof of
Lemmas \ref{Lem:2.1} and \ref{Lem:2.2},
see \cite{BTW}.
\subsection{Decay property of the solution of heat equations}
The following Lemma is also well-known as 
the decay property and
approximation formula 
of the solution of the heat equation. 
For the proof, see e.g. \cite{GGS}.
\begin{lem}
Let $n \ge 1$, $\ell \ge 0$, $k \ge \tilde{k} \ge 0$ and $1 \le r \le q \le \infty$.
Then there exists $C>0$ such that
\begin{equation} \label{eq:2.3}
\|\partial_{t}^{\ell} \nabla_{x}^{k} e^{t \Delta} g  \|_{q} 
\le C t^{-\frac{n}{2}(\frac{1}{r} -\frac{1}{q})-\ell-\frac{k-\tilde{k}}{2}} 
\| \nabla_{x}^{k-\tilde{k}} g \|_{r}.
\end{equation}
Moreover, if $g \in L^{1} \cap L^{q}$, 
then it holds that
\begin{equation} \label{eq:2.4}
\lim_{t \to \infty}
t^{\frac{n}{2}(1-\frac{1}{q})+\frac{k}{2}} 
\| \nabla_{x}^{k} (e^{t \Delta} g -m G_{t}) \|_{q}=0,
\end{equation} 
where $m= \displaystyle{\int_{\R^{n}}} g(y) dy$ for $1 \le q \le \infty$.
\end{lem} 
%


\subsection{Useful formula}
In this subsection, 
we recall useful estimates to show results in this paper.
The following well-known estimate is frequently used to obtain time decay estimates.  
\begin{lem} \label{Lem:2.4}
Let $n \ge 1$, $k \ge 0$ and $1 \le r \le 2$. 
Then there exists a constant $C>0$ such that  
\begin{equation} \label{eq:2.5}
\| |\xi|^{k} e^{-(1+t)|\xi|^{2}} \|_{r} 
\le C (1+ t)^{-\frac{n}{2r}- \frac{k}{2}}.
\end{equation}
\end{lem}
The next lemma is also useful to compute the decay order 
of the nonlinear term in the integral equation.

\begin{lem} \label{Lem:2.5} 
{\rm (}i{\rm )} Let $a>0$ and 
$b>0$ with $\max\{ a, b \} >1$. 
There exists a constant $C$ depending only on $a$ and $b$ such that for $t \geq 0$ it is true that
\begin{equation}
\int_{0}^{t} (1+t-s)^{-a} (1+s)^{-b} ds
\le 
C(1+t)^{-\min\{a, b \}}.  \label{eq:2.6}\\
\end{equation}
{\rm (}ii{\rm )} Let $1 > a \geq 0$, $b > 0$ and $c>0$. 
There exists a constant $C$, which is independent of $t$ such that  for $t \geq 0$ it holds that
\begin{equation}
\int_{0}^{t} e^{-c(t-s)} (t-s)^{- a}(1+s)^{- b} ds
\le C(1+t)^{-b}.  \label{eq:2.7}
\end{equation} 
\end{lem}

The proof of Lemma \ref{Lem:2.5} is well-known (see e.g. \cite{S}).

\section{Point-wise estimates in the Fourier space}
In this section, 
we show point-wise estimates of the Fourier multipliers,
which are important to obtain linear estimates in the next section.
Now, we recall the Fourier multiplier expression of the evolution operator to the linear problem.
According to the notation of \cite{IS} and \cite{IT}
we define the Fourier multipliers $\K_{0}(t, \xi)$ and $\K_{1}(t,\xi)$
as  
\begin{equation*}
\begin{split}
\K_{0}(t, \xi) 
& := 
\frac{-\lambda_{-} e^{\lambda_{+}t } + \lambda_{+}e^{\lambda_{-}t } }{\lambda_{+} - \lambda_{-}}
= 
\dfrac{e^{-t|\xi|^{2}} -|\xi|^{2} e^{-t}}{1-|\xi|^{2}} , \\
\K_{1}(t, \xi) 
& := 
\frac{-e^{\lambda_{-}t } + e^{\lambda_{+}t } }{\lambda_{+} - \lambda_{-}} = \dfrac{e^{-t|\xi|^{2}} -e^{-t}}{1-|\xi|^{2}},  
\end{split}
\end{equation*}
and 
the evolution operators $K_{0}(t)g$ and $K_{1}(t) g$ to problem \eqref{eq:1.1} by 
\begin{equation} \label{eq:3.1}
\begin{split}
K_{j}(t)g 
:= \mathcal{F}^{-1} [\K_{j}(t, \xi) \hat{g}]
\end{split}
\end{equation}
for $j=0,1$, where $\lambda_{\pm}$ are the characteristic roots computed through the corresponding algebraic equations (see Section 3 of \cite{IT})
\[\lambda^{2} + (1+\vert\xi\vert^{2})\lambda + \vert\xi\vert^{2} = 0.\]
Moreover, using the cut-off functions $\chi_{k}$ ($k=L,M,H$), 
we introduce the ``localized'' evolution operators by
\begin{equation} \label{eq:3.2}
\begin{split}
K_{jk}(t)g 
:= \mathcal{F}^{-1} [\K_{jk}(t, \xi) \hat{g}],
\end{split}
\end{equation}
where $\K_{jk}(t, \xi) := \K_{j}(t, \xi) \chi_{k}$, for $j=0,1$, $k=L,M,H$.  
%

\subsection{Estimates for the low frequency parts}
We begin with the following point-wise estimates on small $|\xi|$ region in the Fourier space. 
\begin{lem} \label{Lem:3.1}
Let $n \ge 1$ be an integer and $|\xi| \le 1/2$.
Then there exists a constant $C>0$ such that
\begin{align}
& | e^{-t|\xi|^{2}} -e^{-t}|\xi|^{2} | \le Ce^{-(1+t)|\xi|^{2}}, \label{eq:3.3} \\
& |\nabla_{\xi} (e^{-t|\xi|^{2}} -e^{-t}|\xi|^{2}) | \le Ce^{-(1+t)|\xi|^{2}}(1+t)|\xi|,  \label{eq:3.4} \\ 
& |\nabla^{2}_{\xi} (e^{-t|\xi|^{2}} -e^{-t}|\xi|^{2}) | \le Ce^{-(1+t)|\xi|^{2}}(1+t+t^{2} |\xi|^{2})  \label{eq:3.5}.
\end{align}
\end{lem}
\begin{proof}
The proof is straightforward.
Noting $|\xi| \le \displaystyle{\frac{1}{2}}$, we easily see that
\begin{equation*}
\begin{split}
| e^{-t|\xi|^{2}} -e^{-t}|\xi|^{2} |
\le C (e^{-t|\xi|^{2}} +e^{-t} )
\le Ce^{-(1+t)|\xi|^{2}},
\end{split}
\end{equation*}
and
\begin{equation*}
\begin{split}
|\nabla_{\xi} (e^{-t|\xi|^{2}} -e^{-t}|\xi|^{2}) |
\le C e^{-t|\xi|^{2}}t |\xi| + Ce^{-t}|\xi| 
\le Ce^{-(1+t)|\xi|^{2}}(1+t)|\xi|,
\end{split}
\end{equation*}
which prove the estimates \eqref{eq:3.3} and \eqref{eq:3.4}, respectively.
Finally we show the estimate \eqref{eq:3.5}.
Taking the second derivative and using $|\xi| \le \displaystyle{\frac{1}{2}}$ again, 
we have
\begin{equation*}
\begin{split}
|\nabla^{2}_{\xi} (e^{-t|\xi|^{2}} -e^{-t}|\xi|^{2}) | 
& = 2 |\nabla_{\xi} (e^{-t|\xi|^{2}}t \xi-e^{-t} \xi) | \\
& \le C (e^{-t|\xi|^{2}}(|t \xi|^{2}+ t) + e^{-t})  \\
& \le Ce^{-(1+t)|\xi|^{2}}(1+t+t^{2} |\xi|^{2}),  
\end{split}
\end{equation*}
which is the desired estimate \eqref{eq:3.5}, and the proof is complete.
\end{proof}
The following estimate is useful to obtain the decay property and the large time behavior of the evolution operator 
$K_{1}(t)g$.
\begin{lem} \label{Lem:3.2}
Let $n \ge 1$ be an integer and $|\xi| \le 1/2$.
Then there exists constant $C>0$ such that
\begin{align}
& | e^{-t|\xi|^{2}} -e^{-t} | \le Ce^{-(1+t)|\xi|^{2}}, \label{eq:3.6} \\
& |\nabla_{\xi} (e^{-t|\xi|^{2}} -e^{-t}) | \le Ce^{-(1+t)|\xi|^{2}}t|\xi|,  \label{eq:3.7} \\ 
& |\nabla^{2}_{\xi} (e^{-t|\xi|^{2}} -e^{-t}) | \le Ce^{-(1+t)|\xi|^{2}}(t+t^{2} |\xi|^{2})  \label{eq:3.8}.
\end{align}
\end{lem}
\begin{proof}
The proof is standard. We have \eqref{eq:3.6} by similar arguments to \eqref{eq:3.3}. 
When $k>0$, by applying
$\nabla^{k}_{\xi} (e^{-t|\xi|^{2}} -e^{-t})=\nabla^{k}_{\xi} e^{-t|\xi|^{2}}$, 
\eqref{eq:3.7} and \eqref{eq:3.8} can be derived.
\end{proof}
As an easy consequence of Lemmas \ref{Lem:3.1} and \ref{Lem:3.2}, 
we arrive at the point-wise estimates for the Fourier multipliers with small $|\xi|$.
\begin{cor} \label{Cor:3.3}
Under the assumptions as in Lemmas {\rm \ref{Lem:3.1}} and Lemma {\rm \ref{Lem:3.2}}, 
it holds that
\begin{align}
& | \K_{jL}(t, \xi)  | \le Ce^{-(1+t)|\xi|^{2}} \chi_{L}, \label{eq:3.9} \\
& |\nabla_{\xi} \K_{jL}(t, \xi)| \le Ce^{-(1+t)|\xi|^{2}}(1+t)|\xi| \chi_{L} + Ce^{-\frac{t}{4}}\chi_{L}', 
\label{eq:3.10}
\\ 
& |\nabla^{2}_{\xi} \K_{jL}(t, \xi) | \le Ce^{-(1+t)|\xi|^{2}}(1+t+t^{2} |\xi|^{2})\chi_{L}
+ Ce^{-\frac{t}{4}} (\chi_{L}'+\chi_{L}'') \label{eq:3.11}
\end{align}
for $j=0,1$.
\end{cor}
\begin{proof}
The estimates \eqref{eq:3.9}, \eqref{eq:3.10} and \eqref{eq:3.11} for $j=1$ 
are shown by the same argument.
Here we only show \eqref{eq:3.11} with $j=0$.
We first note that 
\begin{align} \label{eq:3.12}
|\nabla^{k}_{\xi}(1-|\xi|^{2})^{-1}| \le 
\begin{cases}
& C|\xi|,\ \text{for} \ k=1, \\
& C, \ \text{for integers} \ k \ge 0. 
\end{cases}
\end{align}
In addition, 
it is easy to see that
\begin{equation} \label{eq:3.13}
|\nabla^{k}_{\xi} \K_{0L}(t, \xi)| \le C e^{-\frac{t}{4}}
\end{equation}
on $\supp \chi_{L}' \cup \supp \chi_{L}''$ by \eqref{eq:3.3} - \eqref{eq:3.5} and \eqref{eq:3.12} 
with $k=0,1$.
Thus, a direct calculation, \eqref{eq:3.12}, \eqref{eq:3.13} and Lemma \ref{Lem:3.1} show that 
\begin{equation*}
\begin{split}
|\nabla^{2}_{\xi} \K_{0L}(t, \xi) |
& \le 
C \left|\nabla^{2}_{\xi}
\left(
\frac{e^{-t|\xi|^{2}}- e^{-t}|\xi|^{2}}{1- |\xi|^{2}}
\chi_{L}
\right) \right| \\
& \le  C\chi_{L} |\nabla^{2}_{\xi}(e^{-t|\xi|^{2}}- e^{-t}|\xi|^{2}) |
+  C\chi_{L}|\xi| |\nabla (e^{-t|\xi|^{2}}- e^{-t}|\xi|^{2})| \\
& + C\chi_{L} |e^{-t|\xi|^{2}}- e^{-t}|\xi|^{2}| + Ce^{-\frac{t}{4}}(\chi_{L}'+\chi_{L}'') \\
& \le  C\chi_{L}(1+t+t^{2}|\xi|^{2})e^{-(1+t)|\xi|^{2}}
+  C\chi_{L}|\xi|^{2} e^{-(1+t)|\xi|^{2}} \\
& + C\chi_{L}  e^{-(1+t)|\xi|^{2}} + Ce^{-\frac{t}{4}}(\chi_{L}'+\chi_{L}'') \\
& \le Ce^{-(1+t)|\xi|^{2}}(1+t+t^{2} |\xi|^{2})\chi_{L}+ Ce^{-\frac{t}{4}} (\chi_{L}'+\chi_{L}''),
\end{split}
\end{equation*}
which is the desired estimate \eqref{eq:3.11} with $j = 0$.
The proof of Corollary \ref{Cor:3.3} is now complete.
\end{proof}
The following result plays an important role 
to obtain asymptotic profiles of the evolution operators $K_{0}(t)g$ and $K_{1}(t)g$.
 \begin{cor} \label{Cor:3.4}
Under the same assumption as in Lemmas {\rm \ref{Lem:3.1}} and Lemma {\rm \ref{Lem:3.2}}, 
it holds that
\begin{align}
& | \K_{jL}(t, \xi) -e^{-t}|\xi|^{2} \chi_{L} | \le C|\xi|^{2}e^{-(1+t)|\xi|^{2}} \chi_{L}, \label{eq:3.14} \\
& |\nabla_{\xi} (\K_{jL}(t, \xi)-e^{-t}|\xi|^{2} \chi_{L}) | 
\le Ce^{-(1+t)|\xi|^{2}} |\xi| (1+t|\xi|^{2}) \chi_{L} + Ce^{-\frac{t}{4}}\chi_{L}', 
\label{eq:3.15}
\\ 
& |\nabla^{2}_{\xi} (\K_{jL}(t, \xi)-e^{-t}|\xi|^{2} \chi_{L}) | 
\le Ce^{-(1+t)|\xi|^{2}}(1+t|\xi|^{2}+t^{2} |\xi|^{4})\chi_{L}+ Ce^{-\frac{t}{4}} (\chi_{L}'+\chi_{L}'') \label{eq:3.16}
\end{align}
for $j=0,1$.
\end{cor}
\begin{proof}
We first consider the case $j=0$.
Combining the estimate \eqref{eq:3.9} with $j=1$ and the fact that 
\begin{equation} \label{eq:3.17}
\K_{0L}(t,\xi) - e^{-t}|\xi|^{2} \chi_{L}= |\xi|^{2} \K_{1L}(t,\xi),
\end{equation}
one can get \eqref{eq:3.14} with $j=0$.
In order to show \eqref{eq:3.15} and \eqref{eq:3.16}, 
by using \eqref{eq:3.17} again we see that 
\begin{equation} \label{eq:3.18}
\begin{split}
& |\nabla_{\xi}^{k} 
(\K_{0L}(t,\xi) - e^{-t}|\xi|^{2} \chi_{L}) | \\
& \le 
\begin{cases} 
& C (|\xi| |\K_{1L}(t,\xi)| +|\xi|^{2} |\nabla \K_{1L}(t,\xi)| )\ \text{for}\ k=1, \\
& C (|\K_{1L}(t,\xi)| + |\xi| |\nabla \K_{1L}(t,\xi)|
+ |\xi|^{2} |\nabla^{2} \K_{1L}(t,\xi)| )\ \text{for}\ k=2.
\end{cases}
\end{split}
\end{equation}
Combining \eqref{eq:3.18} and \eqref{eq:3.10} with $j=1$ yields 
the estimate \eqref{eq:3.15} with $j=0$.
We now apply this argument again to \eqref{eq:3.10} with $j=1$ replaced by \eqref{eq:3.11} with $j=1$, 
to obtain the estimate \eqref{eq:3.16} with $j=0$. 
Finally we prove \eqref{eq:3.14} - \eqref{eq:3.16} with $j=1$.
Noting that 
\begin{equation} \label{eq:3.19}
\K_{1L}(t,\xi) - e^{-t}|\xi|^{2} \chi_{L}= \frac{ e^{-t|\xi|^{2}}|\xi|^{2}-e^{-t} }{1-|\xi|^{2}},
\end{equation}
and applying a similar argument to \eqref{eq:3.6}, one gets \eqref{eq:3.14} with $j=1$.
Moreover, using $ \nabla^{k}_{\xi} (e^{-t|\xi|^{2}}|\xi|^{2}-e^{-t}) = \nabla^{k}_{\xi} e^{-t|\xi|^{2}}|\xi|^{2}$ 
for $k>0$,
we can deduce that
\begin{align}
& |\nabla_{\xi} (e^{-t|\xi|^{2}}|\xi|^{2}-e^{-t})| \le C|\xi|(1+ t |\xi|^{2})e^{-t|\xi|^{2}}, \label{eq:3.20} \\
& |\nabla_{\xi}^{2} (e^{-t|\xi|^{2}}|\xi|^{2}-e^{-t})| \le C (1+ t |\xi|^{2} + t^{2} |\xi|^{4})e^{-t|\xi|^{2}}.  \label{eq:3.21} 
\end{align}
Therefore, 
by \eqref{eq:3.14} with $j=1$ and \eqref{eq:3.20}, 
we obtain \eqref{eq:3.15} with $j=1$.
Likewise, we use \eqref{eq:3.14} and \eqref{eq:3.15} with $j=1$ and \eqref{eq:3.21} 
to meet \eqref{eq:3.16} with $j=1$, and the Corollary follows.
\end{proof}
\subsection{Estimates for the middle and high frequency parts}
The following lemma states that the middle part for $\vert\xi\vert$ has a sufficient regularity and decays fast.  
\begin{lem} \label{Lem:3.5}
Let $n \ge 1$ and $k \ge 0$.
Then there exists a constant $C>0$ such that 
\begin{align} \label{eq:3.22}
|\nabla_{\xi}^{k} \K_{jM}(t, \xi) \chi_{M}|
\le C e^{-\frac{t}{4}} \chi_{M}
\end{align}
for $j=0,1$.
\end{lem}
\begin{proof}
The support of the middle part 
$\nabla_{\xi}^{k} \K_{jM}(t, \xi) \chi_{M} $
is compact and does not contain a neighborhood of the origin $\xi=0$.
Therefore, we can estimate the polynomial of $|\xi|$ by a constant.
This implies the desired estimate \eqref{eq:3.22}, and the proof is now complete.
\end{proof}
The rest part of this subsection is devoted to the point-wise estimates for the high frequency parts $K_{jH}(t)g$ for $j=0,1$.
\begin{lem} \label{Lem:3.6}
Let $n=1,2,3$, $\varepsilon > 0$ and $\alpha \in \{ 2, \displaystyle{\frac{n}{2}} + \varepsilon \}$.
Then it holds that 
\begin{equation} \label{eq:3.23}
\nabla^{k}_{\xi} 
\left( 
\frac{|\xi|^{2-\alpha}}{1-|\xi|^{2}} \chi_{H}
\right)
\in L^{2}(\R^{n})
\end{equation}
for $k=0,1,2$.
\end{lem}
\begin{proof}
It is easy to see that 
\begin{equation} \label{eq:3.24}
\nabla^{k}_{\xi} \left( \frac{|\xi|^{2-\alpha}}{1-|\xi|^{2}} \right) = O(|\xi|^{-\alpha-k})
\end{equation}
as $|\xi| \to \infty$ and $2(-\alpha-k )< -n$.
Moreover the support of $\displaystyle{\frac{|\xi|^{2-\alpha}}{1-|\xi|^{2}} \chi_{H}}$ 
does not have a neighborhood of $|\xi|=1$.  
Summing up these facts, we can assert \eqref{eq:3.23}, and the proof is complete.
\end{proof}
\begin{lem} \label{Lem:3.7}
Let $n \ge 1$ and $|\xi| \ge 3$.
Then there exists a constant $C>0$ such that
\begin{align}
& |\nabla_{\xi} (e^{-t|\xi|^{2}} -e^{-t}) | \le Ce^{-t|\xi|^{2}}t|\xi|,  \label{eq:3.25} \\ 
& |\nabla^{2}_{\xi} (e^{-t|\xi|^{2}} -e^{-t}) | \le Ce^{-t|\xi|^{2}}(t+t^{2} |\xi|^{2})  \label{eq:3.26}.
\end{align}
\end{lem}
\begin{proof}
Applying $\nabla^{k}_{\xi} (e^{-t|\xi|^{2}} -e^{-t})=\nabla^{k}_{\xi} e^{-t|\xi|^{2}}$ again, we easily have 
Lemma \ref{Lem:3.7}. 
\end{proof}
\begin{cor} \label{Cor:3.8}
Under the same assumptions as in Lemma {\rm \ref{Lem:3.6}}, 
there exists a constant $C>0$ such that 
\begin{align} 
& \left| \K_{1H}(t,\xi) \right| 
\le Ce^{-t} |\xi|^{-2} \chi_{H}, \label{eq:3.27} \\
& \left| \nabla_{\xi} \K_{1H}(t,\xi) \right| 
\le C e^{-t} (t e^{-t|\xi|^{2}} + |\xi|^{-3} )\chi_{H}+ C e^{-t} \chi_{H}', \label{eq:3.28} \\
& \left| \nabla^{2}_{\xi} \K_{1H}(t,\xi) \right| 
\le  C  e^{-t}
(t+t^{2})  e^{-t|\xi|^{2}} \chi_{H}
+   e^{-\frac{t}{2}} |\xi|^{-2} + C e^{-\frac{t}{2}}(\chi_{H}' + \chi_{H}''). \label{eq:3.29}
\end{align}
\end{cor}
\begin{proof}
Since \eqref{eq:3.27} - \eqref{eq:3.29} are shown by the similar way,
we only check the validity of \eqref{eq:3.29}. 
We first note that
\begin{align} \label{eq:3.30}
\left|\K_{j}(t,\xi) \chi_{H}' \right|
+
\left| \nabla_{\xi} \K_{j}(t,\xi) \chi_{H}' \right| +\left| \K_{j}(t,\xi) \chi_{H}'' \right|
\le C e^{-\frac{t}{2}}  (\chi_{H}' + \chi_{H}'')
\end{align}
for $j=0,1$.
Indeed, the support of $\chi_{H}'$ and $\chi_{H}''$ is compact 
and does not include a neighborhood of $\xi=0$.   
So, the direct calculation and \eqref{eq:3.24} - \eqref{eq:3.26} show 
\begin{equation} \label{eq:3.31}
\begin{split}
& |\nabla^{2}_{\xi}  \K_{1}(t,\xi)| \\
& \le C |\nabla^{2}_{\xi} (e^{-t|\xi|^{2}})| |\xi|^{-2} 
+ 
C  |\nabla_{\xi}  e^{-t|\xi|^{2}}||\nabla_{\xi} (1- |\xi|^{2})^{-1}| 
+ 
C  e^{-t|\xi|^{2} } |\nabla^{2}_{\xi} (1- |\xi|^{2})^{-1}|\\
&  \le C e^{-t}e^{-t|\xi|^{2}} 
\left\{ 
(t+t^{2}|\xi|^{2}) |\xi|^{-2} 
+  |\xi|^{-3} t |\xi|+   |\xi|^{-4} )
\right\} \\
& \le C e^{-t} \{
(t+t^{2})  e^{-t|\xi|^{2}} 
+  |\xi|^{-2} t +   |\xi|^{-4} \}  \\
& \le C e^{-t}
(t+t^{2})  e^{-t|\xi|^{2}} 
+   e^{-\frac{t}{2}} |\xi|^{-2}
\end{split}
\end{equation}
for $|\xi| \ge 3$.
Thus combining \eqref{eq:3.30} and \eqref{eq:3.31},  
we see
\begin{equation*}
\begin{split}
|\nabla^{2}_{\xi}  \K_{1H}(t,\xi)| 
& \le C \chi_{H}|\nabla^{2} \K_{1}(t,\xi)| + Ce^{-\frac{t}{2}} (\chi_{H}'+ \chi_{H}'')  \\
&  \le C  e^{-t}
(t+t^{2})  e^{-t|\xi|^{2}}  \chi_{H}
+   e^{-\frac{t}{2}} |\xi|^{-2} \chi_{H} + C e^{-\frac{t}{2}}(\chi_{H}' + \chi_{H}''),
\end{split}
\end{equation*}
which is the desired conclusion.
\end{proof}
The following estimates are useful for the estimates for $K_{0H}(t)g$.
\begin{cor} \label{cor:3.9}
Under the same assumptions as in Lemma {\rm \ref{Lem:3.6}}, 
there exists a constant $C>0$ such that 
\begin{equation} \label{eq:3.32}
\begin{split}
& \left|\left(
\frac{e^{-t|\xi|^{2}} - e^{-t}|\xi|^{2-(\frac{n}{2}+\varepsilon)} }{1-|\xi|^{2}} \chi_{H} 
\right) \right|
\le Ce^{-t} |\xi|^{-2} \chi_{H} + C e^{-\frac{t}{2}} |\xi|^{-\frac{n}{2}-\varepsilon-k} \chi_{H},
\end{split}
\end{equation}
\begin{equation} \label{eq:3.33}
\begin{split}
& \left| \nabla_{\xi} \left(
\frac{e^{-t|\xi|^{2}} - e^{-t}|\xi|^{2-(\frac{n}{2}+\varepsilon)} }{1-|\xi|^{2}} \chi_{H} 
\right) \right| \\
& \le  C e^{-t} (t e^{-t|\xi|^{2}} + |\xi|^{-3} )\chi_{H}+ C e^{-t} \chi_{H}'
+ C e^{-\frac{t}{2}} |\xi|^{-\frac{n}{2}-\varepsilon-k} \chi_{H},
\end{split}
\end{equation}
\begin{equation} \label{eq:3.34}
\begin{split}
& \left| \nabla^{2}_{\xi} \left(
\frac{e^{-t|\xi|^{2}} - e^{-t}|\xi|^{2-(\frac{n}{2}+\varepsilon)} }{1-|\xi|^{2}} \chi_{H} 
\right) \right| \\
& \le C  e^{-t}
(t+t^{2})  e^{-t|\xi|^{2}} \chi_{H} 
+   e^{-\frac{t}{2}} |\xi|^{-2} \chi_{H} + C e^{-\frac{t}{2}}(\chi_{H}' + \chi_{H}'')
+ C e^{-\frac{t}{2}} |\xi|^{-\frac{n}{2}-\varepsilon-k} \chi_{H}.
\end{split}
\end{equation}
\end{cor}
%
\begin{proof}
Observing the fact that  
\begin{equation} \label{eq:3.35}
\begin{split}
& \left|\nabla^{k}_{\xi} \left(
\frac{e^{-t|\xi|^{2}} - e^{-t}|\xi|^{2-(\frac{n}{2}+\varepsilon)} }{1-|\xi|^{2}} \chi_{H} 
\right) \right| \\
& \le 
\left|\nabla^{k}_{\xi}  \left(
\frac{e^{-t|\xi|^{2}} }{1-|\xi|^{2}} \chi_{H} 
\right) \right|
+e^{-t}
\left|\nabla^{k}_{\xi}  \left(
\frac{|\xi|^{2-(\frac{n}{2}+\varepsilon)} }{1-|\xi|^{2}} \chi_{H} 
\right) \right|,
\end{split}
\end{equation}
we see that the first factor in the right hand side of \eqref{eq:3.35} satisfy the following estimates 
\begin{equation} \label{eq:3.36}
\begin{split} 
& \left|
\frac{e^{-t|\xi|^{2}} }{1-|\xi|^{2}} \chi_{H} \right|
\le Ce^{-t} |\xi|^{-2} \chi_{H}, \\
& \left|\nabla_{\xi} \left(
\frac{e^{-t|\xi|^{2}} }{1-|\xi|^{2}} \chi_{H} 
\right) \right|
\le  C e^{-t} (t e^{-t|\xi|^{2}} + |\xi|^{-3} )\chi_{H}+ C e^{-t} \chi_{H}', \\
& \left|\nabla^{2}_{\xi} \left(
\frac{e^{-t|\xi|^{2}} }{1-|\xi|^{2}} \chi_{H} 
\right) \right| 
\le C  e^{-t}
(t+t^{2})  e^{-t|\xi|^{2}} \chi_{H} 
+   e^{-\frac{t}{2}} |\xi|^{-2}\chi_{H}  + C e^{-\frac{t}{2}}(\chi_{H}' + \chi_{H}''), 
\end{split}
\end{equation}
as in Corollary \ref{Cor:3.8}.
Furthermore, by using \eqref{eq:3.24} with $\alpha=\displaystyle{\frac{n}{2}} + \varepsilon$, and \eqref{eq:3.31} with $j=1$,
the second factor in the right hand side of \eqref{eq:3.35}
is estimated as follows
\begin{equation} \label{eq:3.37}
\begin{split}
e^{-t}
\left|\nabla^{k}_{\xi} \left(
\frac{|\xi|^{2-(\frac{n}{2}+\varepsilon)} }{1-|\xi|^{2}} \chi_{H} 
\right) \right| \le C e^{-\frac{t}{2}} |\xi|^{-\frac{n}{2}-\varepsilon-k} \chi_{H}.
\end{split}
\end{equation}
Summing up these estimates \eqref{eq:3.35} - \eqref{eq:3.37},
one can conclude \eqref{eq:3.32} - \eqref{eq:3.34}.
\end{proof}

\section{Linear estimates}
In this section, we shall study an important decay property of the solution $u(t,x)$ to the corresponding linear equation in order to handle with the original semi-linear problem \eqref{eq:1.1} 
\begin{equation} \label{eq:4.1}
\left\{
\begin{split}
& \partial_{t}^{2} u -\Delta u + \partial_{t} u -\Delta \partial_{t} u =0, \quad t>0, \quad x \in \R^{n}, \\
& u(0,x)=u_{0}(x), \quad \partial_{t} u(0,x)=u_{1}(x) , \quad x \in \R^{n}.
\end{split}
\right.
\end{equation}
Our purpose is to show the following proposition, which suggests large time behaviors of the solution to the linear problem above in $L^{1} \cap L^{\infty}$ framework. 
\begin{prop} \label{Prop:4.1}
Let $n = 1,2,3$ and $\varepsilon>0$.
Assume that 
$(u_{0}, u_{1}) \in (W^{\frac{n}{2}+\varepsilon, 1} \cap W^{\frac{n}{2} + \varepsilon, \infty}) \times (L^{1} \cap L^{\infty})$. 
Then, there exists a unique solution $u \in C([0, \infty);L^{1} \cap L^{\infty})$ to problem \eqref{eq:4.1}  such that
\begin{align} \label{eq:4.2}
& \| u(t,\cdot) \|_{L^{q}(\R^{n})} \le C(1+t)^{-\frac{n}{2}(1-\frac{1}{q})},
\end{align}
\begin{align} \label{eq:4.3}
&  \| u(t,\cdot)-\tilde{M} G_{t} \|_{L^{q}(\R^{n})} =o(t^{-\frac{n}{2}(1-\frac{1}{q})})\quad (t \to \infty)
\end{align}
for $q \in [1, \infty]$, where $\tilde{M}= \displaystyle{\int_{\R^{n}}}(u_{0}(y) + u_{1}(y)) dy$.
\end{prop}
\subsection{Decay estimates for ``localized'' evolution operators}
In this subsection, 
we prepare several decay properties of the evolution operators.  
\begin{lem} \label{Lem:4.2}
Let $n =1,2,3$, $1 \le r \le q \le \infty$.
Then there exists a constant $C>0$ such that 
\begin{align} \label{eq:4.4}
\|K_{jL}(t) g \|_{q}
\le
C(1+t)^{-\frac{n}{2}(\frac{1}{r}-\frac{1}{q}) }
\| g \|_{r}
\end{align}
for $j=0,1$.
\end{lem}
\begin{lem} \label{Lem:4.3}
Let $n =1,2,3$, $\varepsilon>0$ and $1 \le r \le  q \le \infty$.
Then there exists a constant $C>0$ such that 
\begin{equation} \label{eq:4.5}
\|K_{0H}(t) g \|_{q}
\le
\begin{cases}
& Ce^{-\frac{t}{2}} \| |\nabla_{x}|^{\frac{n}{2}+ \varepsilon}g \|_{r} \ \text{for}\ n=1, \\
& Ce^{-\frac{t}{2}} \| |\nabla_{x}|^{\frac{n}{2}+ \varepsilon}g \|_{q} \ \text{for}\ n=2,3,
\end{cases}
\end{equation}
\begin{align} \label{eq:4.6}
\|K_{1H}(t) g \|_{q}
\le
\begin{cases}
& Ce^{-\frac{t}{2}}
\| g \|_{r} \text{for}\ n=1,  \\
& C  e^{-\frac{t}{2}} \|  g \|_{q}\ \text{for}\ n=2,3, 
\end{cases}
\end{align}
and
\begin{align} \label{eq:4.7}
\|K_{jM}(t) g \|_{q}
\le
Ce^{-\frac{t}{2}}
\| g \|_{r} \ \text{for}\ j=0,1.
\end{align}
\end{lem}
\begin{proof}[Proof of Lemma \ref{Lem:4.2}]
To show \eqref{eq:4.4}, 
it is sufficient to show that 
\begin{align} 
& \|K_{jL}(t) g \|_{\infty}
\le
C(1+t)^{-\frac{n}{2}}
\| g \|_{1}, \label{eq:4.8} \\
& \|K_{jL}(t) g \|_{q}
\le
C \| g \|_{q}, \label{eq:4.9} 
\end{align}
for $1 \le q \le \infty$.
Indeed, once we have \eqref{eq:4.8} and \eqref{eq:4.9}, the Riesz-Thorin complex interpolation theorem yields 
\eqref{eq:4.4}. 
So, we first show \eqref{eq:4.8}.
By the Hausdorff-Young inequality and \eqref{eq:2.5}, we see that 
\begin{align*}
\|K_{jL}(t) g \|_{\infty} 
& \le C \| \K_{jL}(t, \xi) \hat{g} \|_{1}
\le \| \K_{jL}(t) \|_{1} 
\| \hat{g} \|_{\infty} \\
& \le \| e^{-(1+t)|\xi|^{2}} \|_{1}
\| g \|_{1} = C(1+t)^{-\frac{n}{2}} \|g \|_{1},
\end{align*}
which show the desired estimate \eqref{eq:4.8}.
Next, we prove \eqref{eq:4.9} by applying \eqref{eq:2.2}. 
Then by using \eqref{eq:3.9} - \eqref{eq:3.11} and \eqref{eq:2.5}, 
we can assert the upper bounds of $\| \nabla^{k}_{\xi} \K_{jL}(t) \|_{2}$ for $k=0,1,2$  
as follows:
\begin{align} \label{eq:4.10}
& \| \nabla^{k}_{\xi} \K_{jL}(t) \|_{2}
\le C (1+t)^{-\frac{n}{4}+\frac{k}{2}}. 
\end{align}
Therefore 
for $n=1$,
we apply \eqref{eq:4.10} with $k=0,1$ and \eqref{eq:2.2} with $s=1$ to have
\begin{equation} \label{eq:4.11}
\begin{split}
M_{\infty}(\K_{jL}(t))
& \le C \| \K_{jL}(t) \|_{2}^{1-\frac{1}{2}} \| \K_{jL}(t) \|_{\dot{H}^{1}}^{\frac{1}{2}} \\
& \le C \| \K_{jL}(t) \|_{2}^{1-\frac{1}{2}} \| \nabla_{\xi} \K_{jL}(t) \|_{2}^{\frac{1}{2}} \\
& \le C (1+t)^{-\frac{1}{4}} (1+t)^{-\frac{1}{4}+\frac{1}{2}} \le C.
\end{split}
\end{equation}
On the other hand, 
for $n=2,3$, 
we use \eqref{eq:4.10} with $k=0,2$ and \eqref{eq:2.2} with $s=2$ to see
\begin{equation} \label{eq:4.12}
\begin{split}
M_{\infty}(\K_{jL}(t))
& \le C \| \K_{jL}(t) \|_{2}^{1-\frac{n}{4}} \| \K_{jL}(t) \|_{\dot{H}^{2}}^{\frac{n}{4}} \\
& \le C \| \K_{jL}(t) \|_{2}^{1-\frac{n}{4}} \| \nabla_{\xi}^{2} \K_{jL}(t) \|_{2}^{\frac{n}{4}} \\
& \le C (1+t)^{-\frac{n}{4}(1-\frac{n}{4})} (1+t)^{\frac{n}{4} (-\frac{n}{4}+1)} \le C.
\end{split}
\end{equation}
By combining \eqref{eq:4.11}, \eqref{eq:4.12} and \eqref{eq:2.1} one can obtain  
\begin{equation*}
\begin{split}
M_{q}(\K_{jL}(t)) \le M_{\infty}(\K_{jL}(t)) \le C
\end{split}
\end{equation*}
for $1 \le q \le \infty$, which proves the desired estimate \eqref{eq:4.9} by definition of $M_{q}$. 
\end{proof}
\begin{proof}[Proof of Lemma \ref{Lem:4.3}]
%
Firstly, we remark that \eqref{eq:4.5} and \eqref{eq:4.6} can be derived by the same idea.
Hence we only check \eqref{eq:4.6}.
As in the proof of Lemma \ref{Lem:4.2},
we only need to show 
\begin{align} 
& \|K_{1H}(t) g \|_{\infty}
\le
Ce^{-t}
\| g \|_{1}, \label{eq:4.13} 
\end{align}
for $n=1$ and 
\begin{align} 
\|K_{1H}(t) g \|_{q}
\le
Ce^{-\frac{t}{2}} \| g \|_{q}, \label{eq:4.14} 
\end{align}
for $1 \le q \le \infty$ and $n=1,2,3$. 
\noindent
For $n=1$, 
the Hausdorff-Young inequality and \eqref{eq:3.27}
yield 
\begin{equation*}
\begin{split}
\|K_{1H}(t) g \|_{\infty}
\le \| \K_{1H}(t, \xi) \hat{g} \|_{1}
\le Ce^{-t} \| |\xi|^{-2} \chi_{H} \|_{1} \| \hat{g} \|_{\infty} \le Ce^{-t} \|g \|_{1},  
\end{split}
\end{equation*}
since $|\xi|^{-2} \chi_{H} \in L^{1}(\R)$, which is the desired estimate \eqref{eq:4.13}. 
In order to show \eqref{eq:4.14}, 
we again apply the same argument as \eqref{eq:4.9}.
Indeed, by \eqref{eq:3.27} - \eqref{eq:3.29}, 
we see 
\begin{equation} \label{eq:4.15}
\begin{split}
\| \nabla^{k}_{\xi} \K_{1H}(t, \xi) \|_{2}
\le C e^{-\frac{t}{2}} 
\end{split}
\end{equation}
for $k=0,1,2$.
Here we have just used the fact that 
\begin{equation*}
e^{-t} t \| e^{-t|\xi|^{2}} \|_{2} \le Ce^{-t} t^{1-\frac{n}{4}} \le Ce^{-\frac{t}{2}},
\end{equation*}
since $1-\displaystyle{\frac{n}{4}}>0$ for $n=1,2,3$.
Therefore,  
we apply \eqref{eq:4.15} with $k=0,1$, \eqref{eq:2.1} and \eqref{eq:2.2} with $s=1$ to have
\begin{equation} \label{eq:4.16}
\begin{split}
M_{q}(\K_{1H}(t)) \le M_{\infty}(\K_{1H}(t))
& \le C \| \K_{1H}(t) \|_{2}^{1-\frac{1}{2}} \| \K_{1H}(t) \|_{\dot{H}^{1}}^{\frac{1}{2}} \\
& \le C \| \K_{1H}(t) \|_{2}^{1-\frac{1}{2}} \| \nabla_{\xi} \K_{1H}(t) \|_{2}^{\frac{1}{2}} \\
& \le C e^{-\frac{t}{2}},
\end{split}
\end{equation}
for the case $n=1$, and in the case when $n = 2,3$, by \eqref{eq:4.15} with $k=0,2$, \eqref{eq:2.1} and \eqref{eq:2.2} with $s=2$ one can find that
\begin{equation} \label{eq:4.17}
\begin{split}
M_{q}(\K_{1H}(t)) \le M_{\infty}(\K_{1H}(t))
& \le C \| \K_{1H}(t) \|_{2}^{1-\frac{n}{4}} \| \K_{1H}(t) \|_{\dot{H}^{2}}^{\frac{n}{4}} \\
& \le C \| \K_{1H}(t) \|_{2}^{1-\frac{n}{4}} \| \nabla_{\xi}^{2} \K_{1H}(t) \|_{2}^{\frac{n}{4}} \\
& \le C e^{-\frac{t}{2}}.
\end{split}
\end{equation}
%
%
By definition of $M_{q}$, with the help of \eqref{eq:4.16} and \eqref{eq:4.17},  
we obtain the desired estimate \eqref{eq:4.14} for $n = 1,2,3$.

%
Finally, we check \eqref{eq:4.7}.  
The proof of \eqref{eq:4.7} is immediate.
Indeed, we now apply the argument for \eqref{eq:4.4}, 
with \eqref{eq:4.10} replaced by \eqref{eq:3.22} 
to obtain \eqref{eq:4.7}, and the proof of Lemma \ref{Lem:4.3} is now complete.

\end{proof}
\subsection{Asymptotic behavior of the low frequency part}
In this subsection, we state that the evolution operators 
$\K_{jL}(t) g$ for $j=0,1$ are well-approximated by the solution of the heat equation 
in the small $|\xi|$ region.
\begin{lem} \label{Lem:4.4}
Let $n =1,2,3$, $1 \le r \le q \le \infty$.
Then there exists a constant $C>0$ such that 
\begin{align} \label{eq:4.18}
\|K_{jL}(t)g -e^{t \Delta} (\check{\chi}_{L} \ast g) \|_{q}
\le
C(1+t)^{-\frac{n}{2}(\frac{1}{r}-\frac{1}{q})-1 }
\| g \|_{r}
\end{align}
for $j=0,1$.
\end{lem}
\begin{proof}
For the proof, we again apply the similar argument to the proof of Lemma \ref{Lem:4.2}.
Namely, we claim that  
\begin{align} 
& \|K_{jL}(t) g -e^{t \Delta} (\check{\chi}_{L} \ast g) \|_{\infty}
\le
C(1+t)^{-\frac{n}{2}-1}
\| g \|_{1}, \label{eq:4.19} \\
& \|K_{jL}(t) g -e^{t \Delta} (\check{\chi}_{L} \ast g) \|_{q}
\le
C (1+t)^{-1} \| g \|_{q}, \label{eq:4.20} 
\end{align}
for $1 \le q \le \infty$.
Here we recall that \eqref{eq:4.19}, \eqref{eq:4.20} and the 
Riesz-Thorin interpolation theorem show \eqref{eq:4.18}. 
Therefore it suffices to prove \eqref{eq:4.19} and \eqref{eq:4.20} in order to get \eqref{eq:4.18}.

We first show \eqref{eq:4.19}.
The Hausdorff - Young inequality, \eqref{eq:3.14} and \eqref{eq:2.5} with $k=2$ and $r=1$ 
show 
\begin{equation*}
\begin{split} 
\|K_{jL}(t) g -e^{t \Delta} (\check{\chi}_{L} \ast g) \|_{\infty}
& \le C \|(\K_{jL}(t) -e^{-t|\xi|^{2}} \chi_{L}) \hat{g} \|_{1} \\
& \le C \| \K_{jL}(t) -e^{-t|\xi|^{2}} \chi_{L} \|_{1} \|\hat{g}  \|_{\infty} \\
& \le C \| |\xi|^{2} e^{-(1+t)|\xi|^{2}} \chi_{L}\|_{1} \| g \|_{1}
\le C(1+t)^{-\frac{n}{2}-1} \|g \|_{1}, 
\end{split}
\end{equation*}
which is the desired estimate \eqref{eq:4.19}.

Next, we prove \eqref{eq:4.20}.
By observing \eqref{eq:3.14} - \eqref{eq:3.16} and \eqref{eq:2.5},
we get 
\begin{equation} \label{eq:4.21}
\begin{split}
\| \nabla^{k}_{\xi} (\K_{jL}(t) -e^{-t|\xi|^{2}} \chi_{L} )\|_{2}
\le C(1+t)^{-\frac{n}{4}-1+\frac{k}{2}}
\end{split}
\end{equation}
for $k=0,1,2$.

In order to check  \eqref{eq:4.20} for the case $n=1$, we apply \eqref{eq:2.2} with $s=1$ and \eqref{eq:4.21} with $k=0,1$
to get 
\begin{equation} \label{eq:4.22}
\begin{split}
M_{\infty}(\K_{jL}(t) -e^{-t|\xi|^{2}} \chi_{L})
& \le C \| \K_{jL}(t) -e^{-t|\xi|^{2}} \chi_{L} \|_{2}^{1-\frac{1}{2}} 
\| \K_{jL}(t) -e^{-t|\xi|^{2}} \chi_{L} \|_{\dot{H}^{1}}^{\frac{1}{2}} \\
& \le C \| \K_{jL}(t) -e^{-t|\xi|^{2}} \chi_{L} \|_{2}^{1-\frac{1}{2}} 
\| \nabla_{\xi} (\K_{jL}(t) -e^{-t|\xi|^{2}} \chi_{L}) \|_{2}^{\frac{1}{2}} \\
& \le C (1+t)^{\frac{1}{2}(-\frac{1}{4}-1)} (1+t)^{\frac{1}{2}(-\frac{1}{4}-\frac{1}{2})} 
\le C (1+t)^{-1}.
\end{split}
\end{equation}
Namely, we have arrived at \eqref{eq:4.20} with $n = 1$ since combining \eqref{eq:2.1} and \eqref{eq:4.22} gives \eqref{eq:4.20}.

In the case when $n=2,3$, 
we use \eqref{eq:4.21} with $k=0,2$ and \eqref{eq:2.2} with $s=2$ to obtain
\begin{equation*} 
\begin{split}
M_{\infty}(\K_{jL}(t)-e^{-t|\xi|^{2}} \chi_{L} )
& \le C \| \K_{jL}(t) -e^{-t|\xi|^{2}} \chi_{L} \|_{2}^{1-\frac{n}{4}} 
\| \K_{jL}(t) -e^{-t|\xi|^{2}} \chi_{L}  \|_{\dot{H}^{2}}^{\frac{n}{4}} \\
& \le C \| \K_{jL}(t) -e^{-t|\xi|^{2}} \chi_{L} \|_{2}^{1-\frac{n}{4}} 
\| \nabla_{\xi}^{2} (\K_{jL}(t) -e^{-t|\xi|^{2}} \chi_{L}  )\|_{2}^{\frac{n}{4}} \\
& \le C (1+t)^{(-\frac{n}{4}-1)(1-\frac{n}{4})} (1+t)^{-\frac{n}{4} \frac{n}{4}} = C(1+t)^{-1}.
\end{split}
\end{equation*}
That is, 
$M_{q}(\K_{jL}(t) -e^{-t|\xi|^{2}} \chi_{L} ) \le M_{\infty}(\K_{jL}(t)-e^{-t|\xi|^{2}} \chi_{L} ) \le C(1+t)^{-1}$
for $1 \le q \le \infty$ by \eqref{eq:2.1} again.
This shows \eqref{eq:4.10} with $n = 2,3$, which proves Lemma \ref{Lem:4.4}. 
\end{proof}
\subsection{Proof of Proposition \ref{Prop:4.1}}
In this subsection, we shall prove Proposition \ref{Prop:4.1}. 

We start with the observation that the results obtained in previous subsections guarantee the decay property and large time behavior of the evolution operators $K_{0}(t)$ and $K_{1}(t)$.

\begin{cor} \label{Cor:4.5}
Let $n=1,2,3$, $\varepsilon>0$ and $1 \le r \le q \le \infty$.
Then there exists a constant $C>0$ such that 
\begin{align}  \label{eq:4.23}
& \| K_{0}(t) g \|_{q}
\le
C(1+t)^{-\frac{n}{2}(\frac{1}{r}-\frac{1}{q})}
\| g \|_{r} 
+ 
Ce^{-\frac{t}{2}} \| |\nabla_{x}|^{\frac{n}{2}+\varepsilon} g \|_{q},\\ 
& \|K_{1}(t) g \|_{q}
\le
C(1+t)^{-\frac{n}{2}(\frac{1}{r}-\frac{1}{q})}
\| g \|_{r} 
+ 
Ce^{-\frac{t}{2}} \| g \|_{q},\label{eq:4.24} \\
& \| (K_{0}(t) - e^{t \Delta}) g \|_{q}
\le
C(1+t)^{-\frac{n}{2}(\frac{1}{r}-\frac{1}{q})-1 }
\| g \|_{r} 
+ 
Ce^{-\frac{t}{2}} \| |\nabla_{x}|^{\frac{n}{2}+\varepsilon} g \|_{q},\label{eq:4.25} \\ 
& \| (K_{1}(t) - e^{t \Delta}) g \|_{q}
\le
C(1+t)^{-\frac{n}{2}(\frac{1}{r}-\frac{1}{q})-1 }
\| g \|_{r} 
+ 
Ce^{-\frac{t}{2}} \| g \|_{q}. \label{eq:4.26}
\end{align}
\end{cor} 
\begin{rem}
We note that under the statement above for $n=1$, we see that
\begin{equation*} 
\begin{split}
& \| K_{1}(t)  g \|_{q}
\le
C(1+t)^{-\frac{1}{2}(\frac{1}{r}-\frac{1}{q})}
\| g \|_{r}, \\
& \| (K_{1}(t) - e^{t \Delta}) g \|_{q}
\le
C(1+t)^{-\frac{1}{2}(\frac{1}{r}-\frac{1}{q})-1 }
\| g \|_{r},
\end{split}
\end{equation*}
since $Ce^{-\frac{t}{2}} \| g \|_{r}$ is estimated by $C(1+t)^{-\frac{1}{2}(\frac{1}{r}-\frac{1}{q})-1 } \| g \|_{r}$.
The same reasoning can be applied to the case $q=r$, namely,
\begin{align} \label{eq:4.27}
& \| K_{1}(t)  g \|_{q}
\le \| g \|_{q}, \\
& \| (K_{1}(t) - e^{t \Delta}) g \|_{q}
\le
C(1+t)^{-1 }
\| g \|_{q}.  \label{eq:4.28} 
\end{align}
\end{rem} 
\begin{proof}
The proof of the estimates \eqref{eq:4.23} - \eqref{eq:4.26} is similar.
Here we only show the proof of \eqref{eq:4.23}.
Combining \eqref{eq:4.4} with $j=0$, \eqref{eq:4.5} and \eqref{eq:4.7} with $j=0$,
and the definition of the localized operators,
we see that  
\begin{align*}
\| K_{0}(t) g \|_{q}
& \le \sum_{k=L,M,H}
\| K_{0K}(t) g \|_{q} \\
& \le C(1+t)^{-\frac{n}{2}(\frac{1}{r}-\frac{1}{q})}
\| g \|_{r} 
+ 
Ce^{-\frac{t}{2}}
\| g \|_{r} 
+ 
Ce^{-\frac{t}{2}}\| |\nabla_{x}|^{\frac{n}{2}+\varepsilon} g \|_{q} \\
& \le 
C(1+t)^{-\frac{n}{2}(\frac{1}{r}-\frac{1}{q})}
\| g \|_{r} 
+ 
Ce^{-\frac{t}{2}}\| |\nabla_{x}|^{\frac{n}{2}+\varepsilon} g \|_{q},
\end{align*}
which show the desired estimate \eqref{eq:4.23}.
This completes the proof of Corollary \ref{Cor:4.5}.
\end{proof}
By combining \eqref{eq:4.25}, \eqref{eq:4.26} and \eqref{eq:2.4}, 
we can assert the approximation formula of the evolution operators $K_{0}(t)$ and $K_{1}(t)$ 
in terms of the heat kernel for large $t$. 

\begin{cor} \label{Cor:4.7}
Let $n=1,2,3$, $\varepsilon>0$ and $(g_{0}, g_{1}) \in (W^{\frac{n}{2}+\varepsilon, 1} \cap  W^{\frac{n}{2}+\varepsilon, q}) 
\times (L^{1} \cap L^{q})$.
Then it is true that
\begin{align}\label{eq:4.29}
\| K_{j}(t) g_{j} - m_{j} G_{t} \|_{q}=o(t^{-\frac{n}{2}(1-\frac{1}{q})}),
\end{align}
as $t \to \infty$ for $j=0,1$, where $m_{j} = \displaystyle{\int_{\R^{n}}} g_{j}(y) dy$.
\end{cor} 
\begin{proof}
For $j=0$, we apply \eqref{eq:4.25} and \eqref{eq:2.4} to get
\begin{align*}
& t^{\frac{n}{2}(1-\frac{1}{q})} \| K_{0}(t) g_{0} - m_{0} G_{t} \|_{q} \\
& \le t^{\frac{n}{2}(1-\frac{1}{q})} \| (K_{0}(t) - e^{t \Delta}) g_{0} \|_{q} 
+ t^{\frac{n}{2}(1-\frac{1}{q})} \| e^{t \Delta} g_{0}  - m_{0} G_{t} \|_{q} \\
& \le
C(1+t)^{-1}
\| g_{0} \|_{1} + Ce^{-\frac{t}{2}} \| |\nabla_{x}|^{\frac{n}{2}+\varepsilon} g \|_{q}
+ t^{\frac{n}{2}(1-\frac{1}{q})} \| e^{t \Delta} g_{0}  - m_{0} G_{t} \|_{q} \\
& \to 0
\end{align*}
as $t \to \infty$, which is the desired estimate \eqref{eq:4.29} with $j=0$.
We now apply this argument with \eqref{eq:4.25} replaced by \eqref{eq:4.26}, to obtain the estimate 
\eqref{eq:4.29} with $j=1$, and Corollary \ref{Cor:4.7} now follows.  
\end{proof}

%
Now, we are in a position to prove Proposition \ref{Prop:4.1} by combining Corollaries 
\ref{Cor:4.5} and \ref{Cor:4.7}.
\begin{proof}[Proof of Proposition \ref{Prop:4.1}]
We recall that the solution to \eqref{eq:4.1} is expressed as 
$
u(t,\cdot) = K_{0}(t)u_{0} +K_{1}(t)u_{1}.
$
Then it follows from \eqref{eq:4.23} and \eqref{eq:4.24} with $r=1$, 
\begin{equation*}
\begin{split}
 \| u(t) \|_{q}
\le \| K_{0}(t) u_{0} \|_{q}
+ \| K_{1}(t) u_{1} \|_{q}
\le C(1+t)^{-\frac{n}{2}(1-\frac{1}{q})},
\end{split}
\end{equation*}
which is the desired estimate \eqref{eq:4.2}.
Also we see at once \eqref{eq:4.3}.
Indeed, 
\eqref{eq:4.25}, \eqref{eq:4.26} with $r=1$ and \eqref{eq:4.29} give
\begin{equation*}
\begin{split}
 \| ( u(t,\cdot) - \tilde{M} G_{t} )\|_{q}
& \le \|(K_{0}(t)-e^{t \Delta}) u_{0} \|_{q}
+ \|(K_{1}(t)-e^{t \Delta}) u_{1} \|_{q} \\
& + \| (e^{t \Delta}(u_{0} +u_{1}) -\tilde{M} G_{t} \|_{q} \\
& \le C (1+t)^{-\frac{n}{2}(1-\frac{1}{q})-1} +o(t^{-\frac{n}{2}(1-\frac{1}{q})})
\end{split}
\end{equation*}
as $t \to \infty$, 
which is the desired estimate \eqref{eq:4.3}.
This proves Proposition \ref{Prop:4.1}.
\end{proof}

%
%
%
%

\section{Existence of global solutions}
\par
This section is devoted to the proof of Theorem \ref{Thm:1.1}. Here we prepare some notation, which will be used soon.
We define the closed subspace of $C([0,\infty); L^{1} \cap L^{\infty})$ as
\begin{equation*}
X:= \{ 
u \in C([0,\infty); L^{1} \cap L^{\infty}) ; \| u \|_{X} \le M
\},
\end{equation*}
where 
\begin{equation*}
\| u \|_{X} :=  \sup_{t \ge 0}
\{ \| u(t) \|_{1} 
+ 
(1+t)^{\frac{n}{2}}
\| u(t) \|_{\infty}
\}
\end{equation*}
and $M>0$ will be determined later.
We also introduce the mapping $\Phi$ on $X$ by 
\begin{equation} \label{eq:5.1}
\Phi[u](t)
:= 
K_{0}(t) u_{0} + K_{1}(t)u_{1} 
+ \int_{0}^{t} K_{1}(t-\tau) f(u)(\tau) d \tau.  
\end{equation}
For simplicity of notation, 
we denote the integral term of \eqref{eq:5.1} by $I[u](t)$:
\begin{equation} \label{eq:5.2}
I[u](t) := \int_{0}^{t} K_{1}(t-\tau) f(u)(\tau) d \tau.
\end{equation}
In this situation, we claim that 
\begin{equation} \label{eq:5.3}
\| \Phi[u] \|_{X} \le M 
\end{equation}
for all $u \in X$ and 
\begin{equation} \label{eq:5.4}
\| \Phi[u]- \Phi[v] \|_{X}\le \frac{1}{2} \| u-v \|_{X}
\end{equation}
for all $u,v \in X$. For the proof of Theorem \ref{Thm:1.1}, it suffices to show \eqref{eq:5.3} and \eqref{eq:5.4}.
Indeed,
once we have \eqref{eq:5.3} and \eqref{eq:5.4},
we see that $\Phi$ is a contraction mapping on $X$.
Therefore it is immediate from the Banach fixed point theorem that $\Phi$ has a unique fixed point in
$X$.
Namely, there exists a unique global solution $u=\Phi[u]$ in $X$
and Theorem \ref{Thm:1.1} can be proved. 
We remark that the linear solution 
$K_{0}(t) u_{0} + K_{1}(t)u_{1}$ is estimated suitably by linear estimates stated in Proposition \ref{Prop:4.1}.
In what follows, we concentrate on estimates for $I[u](t)$ defined by \eqref{eq:5.2}.
Firstly we prepare several estimates of the norms for $f(u)$ and $f(u)-f(v)$, 
which will be used below.

By using the mean value theorem, 
we can see that there exists $\theta \in [0,1]$ such that
\begin{equation*}
f(u) -f(v) = f'(\theta u + (1-\theta) v) (u-v).
\end{equation*}
Therefore, by noting the definition of $\| \cdot \|_{X}$, 
we arrive at the estimate
\begin{equation} \label{eq:5.5}
\begin{split}
\| f(u) - f(v) \|_{1} 
& \le \|  f'(\theta u + (1-\theta) v)\|_{\infty}
\| u-v \|_{1} \\
& \le C \| (\theta u + (1-\theta) v \|_{\infty}^{p-1}
\| u-v \|_{1} \\
& \le C (\| u \|_{\infty}^{p-1} + \| v \|_{\infty}^{p-1})
\| u-v \|_{1} \\
& \le C (1+\tau)^{-\frac{n}{2}(p-1)}(\| u \|_{X}^{p-1} + \| v \|_{X}^{p-1})
\| u-v \|_{X} \\
& \le C (1+\tau)^{-\frac{n}{2}(p-1)} M^{p-1}
\| u-v \|_{X}
\end{split}
\end{equation}
for $u,v \in X$.
By the similar way, we have 
\begin{equation} \label{eq:5.6}
\begin{split}
\| f(u) - f(v) \|_{\infty} 
& \le C (\| u \|_{\infty}^{p-1} + \| v \|_{\infty}^{p-1})
\| u-v \|_{\infty} \\
& \le C (1+\tau)^{-\frac{np}{2}}M^{p-1}
\| u-v \|_{X} 
\end{split}
\end{equation}
for $u,v \in X$.
If we take $v=0$ in \eqref{eq:5.5} and \eqref{eq:5.6}, and if we recall $\| u \|_{X} \le M$, 
we easily see that 
\begin{equation} \label{eq:5.7}
\begin{split}
& \| f(u) \|_{1} \le C (1+\tau)^{-\frac{n}{2}(p-1)}M^{p},\\
& \| f(u) \|_{\infty} \le C (1+\tau)^{-\frac{np}{2}}M^{p}
\end{split}
\end{equation}
for $u \in X$.

Now, by using the above estimates in \eqref{eq:5.7}, let us derive the estimate of $\| I[u](t) \|_{1}$ for $n=1,2,3$.

To begin with, we apply \eqref{eq:4.27} with $q=1$,
\eqref{eq:5.8},
\eqref{eq:2.4} and \eqref{eq:2.5} to have
\begin{equation} \label{eq:5.8}
\begin{split}
 \left\| I[u](t) \right\|_{1} 
& \le \int_{0}^{t} \left\| K_{1}(t- \tau) f(u) \right\|_{1}d \tau \le C \int_{0}^{t} \left\| f(u) \right\|_{1}d \tau  \\
& \le C \| u \|_{X}^{p} \int_{0}^{t} (1+\tau)^{-\frac{n}{2}(p-1)} d \tau \le CM^{p},
\end{split}
\end{equation}
since $-\displaystyle{\frac{n}{2}}(p-1)<-1$ for $p>1+\displaystyle{\frac{2}{n}}$.

Secondly 
by the similar way to \eqref{eq:5.8}, 
we calculate $\| I[u](t) -I[v](t) \|_{1}$ as follows:
\begin{equation} \label{eq:5.9}
\begin{split}
 \left\| I[u](t) -I[v](t) \right\|_{1} 
& \le \int_{0}^{t} \left\| K_{1}(t- \tau) (f(u) -f(v)) \right\|_{1}d \tau \\
& \le C \int_{0}^{t} \left\| f(u)-f(v) \right\|_{1}d \tau \\
& \le C M^{p-1} \| u-v \|_{X} \int_{0}^{t} (1+\tau)^{-\frac{n}{2}(p-1)} d \tau \\
& \le  C M^{p-1} \| u-v \|_{X},
\end{split}
\end{equation}
for $u,v \in X$, where we have just used \eqref{eq:5.5} and \eqref{eq:5.6}.

For the proof of Theorem \ref{Thm:1.1},
it still remains to get the estimates for
$\|\Phi[u](t) \|_{\infty}$ and $\|\Phi[u](t)-\Phi[v](t) \|_{\infty}$.

Now,
in order to obtain the estimate for $\|\Phi[u](t) \|_{\infty}$, 
we split the nonlinear term into two parts:
\begin{equation} \label{eq:5.10}
\begin{split}
 \left\| I[u](t) \right\|_{\infty} 
& \le \int_{0}^{\frac{t}{2}} \left\|K_{1}(t- \tau) f(u) \right\|_{\infty}d \tau 
+\int_{\frac{t}{2}}^{t}  \left\|  K_{1}(t- \tau)f(u) \right\|_{\infty} d \tau \\
& =:J_{1}(t) + J_{2}(t).
\end{split}
\end{equation}
To obtain the estimate of $J_{1}(t)$, we apply \eqref{eq:4.24} with $q=\infty$ and $r=1$ and \eqref{eq:5.7}
to have
\begin{equation} \label{eq:5.11}
\begin{split}
J_{1}(t) 
& \le C \int_{0}^{\frac{t}{2}}  (1+t-\tau)^{-\frac{n}{2}}\left\| f(u) \right\|_{1}d \tau 
+ C \int_{0}^{\frac{t}{2}}  e^{-\frac{t- \tau}{2}} \left\| f(u) \right\|_{\infty}d \tau \\
& \le C  (1+t)^{-\frac{n}{2}} \int_{0}^{\frac{t}{2}} (1+ \tau)^{-\frac{n}{2}(p-1)} d \tau M^{p}
+ C  e^{-\frac{1}{2} t}  \int_{0}^{\frac{t}{2}} (1+ \tau)^{-\frac{np}{2}}  d \tau  M^{p} \\
& \le C  (1+t)^{-\frac{n}{2}} M^{p},
\end{split}
\end{equation}
where we have used the fact that $-\displaystyle{\frac{n}{2}}(p-1)<-1$.

For the term $J_{2}(t)$,
by using \eqref{eq:4.27} with $q=\infty$ and \eqref{eq:5.7}
we obtain  
\begin{equation} \label{eq:5.12}
\begin{split}
J_{2}(t) \le C \int_{\frac{t}{2}}^{t} \left\|f(u) \right\|_{\infty}d \tau
 \le C  
\int_{\frac{t}{2}}^{t} 
 (1+\tau)^{-\frac{np}{2}} 
d \tau M^{p} \le C  (1+t)^{-\frac{np}{2}+1} M^{p},
\end{split}
\end{equation}
where we remark that the power in the right hand side $-\displaystyle{\frac{np}{2}}+1$ is strictly 
smaller than $-\displaystyle{\frac{n}{2}}$ since $-\displaystyle{\frac{np}{2}}+1= -\displaystyle{\frac{n}{2}}(p-1)+1-\displaystyle{\frac{n}{2}}$ and 
$-\displaystyle{\frac{n}{2}}(p-1)<-1$.
By combining \eqref{eq:5.10} - \eqref{eq:5.12}, 
we arrive at  
\begin{equation} \label{eq:5.13}
\begin{split}
\left\| I[u](t) \right\|_{\infty}  
\le 
J_{1}(t) +J_{2}(t)  \le C  (1+t)^{-\frac{n}{2}} M^{p}.
\end{split}
\end{equation}

Next, we estimate $\| \Phi[u](t)-\Phi[v](t)  \|_{\infty}$.
Again, we divide $\left\| I[u](t)-I[v](t) \right\|_{\infty}$ into two parts:
\begin{equation} \label{eq:5.14}
\begin{split}
\left\| I[u](t)-I[v](t) \right\|_{\infty} 
& \le \int_{0}^{\frac{t}{2}} \left\| K_{1}(t- \tau) (f(u)- f(v) ) \right\|_{\infty}d \tau \\
& \  +\int_{\frac{t}{2}}^{t}  
\left\| K_{1}(t- \tau) (f(u)- f(v)) \right\|_{\infty} d \tau \\
& =: J_{3}(t) + J_{4}(t).
\end{split}
\end{equation}
As in the proof of \eqref{eq:5.11}, 
we can deduce that
\begin{equation} \label{eq:5.15}
\begin{split}
J_{3}(t) 
& \le C \int_{0}^{\frac{t}{2}} 
 (1+t-\tau)^{-\frac{n}{2}}\left\| f(u) -f(v) \right\|_{1}d \tau \\
& \ + C \int_{0}^{\frac{t}{2}}  e^{-\frac{t- \tau}{2}} \left\| f(u)-f(v) \right\|_{\infty}d \tau \\
& \le C  (1+t)^{-\frac{n}{2}} \int_{0}^{\frac{t}{2}} (1+ \tau)^{-\frac{n}{2}(p-1)} d \tau M^{p-1} 
\| u-v \|_{X} \\
& + C  e^{-\frac{1}{2} t}  \int_{0}^{\frac{t}{2}} (1+ \tau)^{-\frac{np}{2}}  d \tau M^{p-1} 
\| u-v \|_{X} \\
& \le C  (1+t)^{-\frac{n}{2}}  M^{p-1} 
\| u-v \|_{X},
\end{split}
\end{equation}
where we have used the fact that  $-\displaystyle{\frac{np}{2}}+1<-\displaystyle{\frac{n}{2}}$ again.
In the same manner as \eqref{eq:5.12}, 
we can get
\begin{equation} \label{eq:5.16}
\begin{split}
J_{4}(t) 
& \le C \int_{\frac{t}{2}}^{t}  \left\| f(u) -f(v) \right\|_{\infty}d \tau \\
& \le C 
\int_{\frac{t}{2}}^{t}  (1+ \tau)^{-\frac{np}{2}} d \tau 
M^{p-1} 
\| u-v \|_{X} \\
& \le C  (1+t)^{-\frac{np}{2}+1}M^{p-1} 
\| u-v \|_{X}.
\end{split}
\end{equation}
Thus, \eqref{eq:5.14} - \eqref{eq:5.16} yield 
\begin{equation} \label{eq:5.17}
\begin{split}
\left\| I[u](t) -I[v](t) \right\|_{\infty}  
\le 
J_{3}(t) +J_{4}(t)  \le C  (1+t)^{-\frac{n}{2}} M^{p-1} \| u-v \|_{X}.
\end{split}
\end{equation}
By \eqref{eq:4.2},
\eqref{eq:5.9} and \eqref{eq:5.13},
we deduce that
\begin{equation} \label{eq:5.18}
\begin{split}
\| \Phi[u] \|_{X}
& \le \| K_{0}(t) u_0 +K_{1}(t)u_1 \|_{X}
+ \| I[u] \|_{X} \\
&
\le C_{0}(\| u_0 \|_{W^{\frac{n}{2}+ \varepsilon, 1} \cap W^{\frac{n}{2}+ \varepsilon, \infty}}
+\| u_1 \|_{L^1 \cap L^\infty}) +C_{1}^{p}
\end{split}
\end{equation}
for some $C_{0}>0$ and $C_{1}>0$.

Similar arguments can be applied to $\| \Phi[u] -\Phi[v] \|_{X}$ 
by using \eqref{eq:5.9} and \eqref{eq:5.17},
and then one can assert that
\begin{equation} \label{eq:5.19}
\| \Phi[u] -\Phi[v] \|_{X}
\le \| I[u] -I[v] \|_{X} 
\le C_2 M^{p-1} \| u-v \|_{X}
\end{equation}
for some $C_2>0$.
By choosing $\| u_0 \|_{W^{\frac{n}{2}+ \varepsilon, 1} \cap W^{\frac{n}{2}+ \varepsilon, \infty}}+\| u_1 \|_{L^1 \cap L^{\infty}}$ sufficiently small, we can make sure the validity of the inequality such as
\begin{equation} \label{eq:5.20}
C_{1}M^{p} < \frac{1}{2} M,\quad
C_2 M^{p-1} < \frac{1}{2},
\end{equation}
because of the relation $M = 2 C_{0}(\| u_0 \|_{W^{\frac{n}{2}+ \varepsilon, 1} \cap W^{\frac{n}{2}+ \varepsilon, \infty}}+\| u_1 \|_{L^1 \cap L^{\infty}})$. 
By combining \eqref{eq:5.18}, \eqref{eq:5.19} and \eqref{eq:5.20} one has the desired estimates \eqref{eq:5.3} and \eqref{eq:5.4}, and the proof is now complete.

%
\section{Asymptotic behavior of the solution}
In this section, we show the proof of Theorem \ref{Thm:1.2}.
For the proof of Theorem \ref{Thm:1.2}, 
we prepare slightly general setting.
Here, 
we introduce the function $F=F(t,x)\in L^{1}(0,\infty;L^{1}(\R^{n}))$ satisfying 
\begin{align} 
& \| F(t) \|_{q} \le C (1+t)^{-\frac{n}{2}(p-1)-\frac{n}{2}(1-\frac{1}{q})}, \label{eq:6.1}
\end{align}
for $1 \le q \le \infty$ and $p>1+\displaystyle{\frac{2}{n}}$.
We can now formulate our main statement in this section.
\begin{prop} \label{Prop:6.1}
Let $n \ge 1$ and $p>1+\displaystyle{\frac{2}{n}}$, and assume \eqref{eq:6.1}.
Then it holds that 
\begin{equation} \label{eq:6.2}
\left\| 
\left(
\int_{0}^{t} K_{1}(t-\tau) F(\tau) d \tau 
- \int_{0}^{\infty} \int_{\R^{n}}  
 F(\tau, y) dy d \tau \cdot G_{t}(x) 
\right)
\right\|_{q} = o(t^{-\frac{n}{2}(1-\frac{1}{q})})
\end{equation}
as $t \to \infty$.
\end{prop}
As a first step of the proof of Proposition \ref{Prop:6.1}, 
we split the nonlinear terms into five parts.
Namely, we see that
\begin{equation*}
\begin{split}
& \int_{0}^{t} K_{1}(t-\tau) F(\tau) d \tau 
- \int_{0}^{\infty} \int_{\R^{n}} F(\tau, y) dy d \tau \cdot G_{t}(x) \\
& 
= 
\int_{0}^{\frac{t}{2}}( K_{1}(t-\tau)-e^{(t-\tau)\Delta}) F(\tau) d \tau 
+ \int_{\frac{t}{2}}^{t} K_{1}(t-\tau) F(\tau) d \tau \\
& + \int_{0}^{\frac{t}{2}}(e^{(t-\tau)\Delta} -e^{ t\Delta}) F(\tau) d \tau
 + \int_{0}^{\frac{t}{2}}\left( e^{t\Delta} F(\tau) - \int_{\R^{n}}F(\tau, y) dy \cdot G_{t}(x) \right) d \tau \\
& - \int_{\frac{t}{2}}^{\infty} 
\int_{\R^{n}} F(\tau, y) dy d \tau \cdot G_{t}(x),
\end{split}
\end{equation*}
and here we set each terms as follows:
\begin{equation*}
\begin{split}
A_{1}(t)
& := \int_{0}^{\frac{t}{2}}( K_{1}(t-\tau)-e^{(t-\tau)\Delta}) F(\tau) d \tau, \\
A_{2}(t) 
& := \int_{\frac{t}{2}}^{t} K_{1}(t-\tau) F(\tau) d \tau, \ 
A_{3}(t) := \int_{0}^{\frac{t}{2}}(e^{(t-\tau)\Delta} -e^{ t\Delta}) F(\tau) d \tau, \\
A_{4}(t) &
:= \int_{0}^{\frac{t}{2}}\left( e^{t\Delta} F(\tau) - \int_{\R^{n}}F(\tau, y) dy \cdot G_{t}(x) \right) d \tau \\
A_{5}(t) &:=  - \int_{\frac{t}{2}}^{\infty} 
\int_{\R^{n}} F(\tau, y) dy d \tau \cdot G_{t}(x).
\end{split}
\end{equation*}
In what follows, 
we estimate each $A_{j}(t)$ for $j=1, \cdots, 5$, respectively. 
\begin{lem} \label{Lem:6.2}
Under the same assumptions as in Proposition {\rm \ref{Prop:6.1}},
there exists a constant $C>0$ such that
\begin{equation} \label{eq:6.3}
\| A_{1}(t) \|_{q}
\le
C(1+t)^{-\frac{n}{2}(1-\frac{1}{q})-1},
\end{equation}
\begin{equation} \label{eq:6.4}
\| A_{j}(t) \|_{q} \le Ct^{-\frac{n}{2}(1-\frac{1}{q})-\frac{n}{2}(p-1)+1}\quad(j = 2,5),
\end{equation}
\begin{equation} \label{eq:6.5}
\begin{split}
\| A_{3}(t) \|_{q} 
& \le 
\begin{cases}
& C t^{-\frac{n}{2}(1-\frac{1}{q})-1} \log(2+t), \quad p \ge 1+\frac{4}{n}, \\
& C t^{-\frac{n}{2}(1-\frac{1}{q})-\frac{n}{2}(p-1)+1}, \quad 1+\frac{2}{n}< p <1+\frac{4}{n},
\end{cases}
\end{split}
\end{equation}
\begin{equation} \label{eq:6.6}
\| A_{4}(t) \|_{q} = o(t^{-\frac{n}{2}(1-\frac{1}{q})}),
\end{equation}
as $t \to \infty$ for $1 \le q \le \infty$.
\end{lem}
\begin{proof}
First, we show \eqref{eq:6.3}.
By \eqref{eq:4.26} with $r=1$ and \eqref{eq:6.1} 
we see that 
\begin{equation*}
\begin{split}
\| A_{1}(t) \|_{q} 
& \le 
\int_{0}^{\frac{t}{2}} \| ( K_{1}(t-\tau)-e^{(t-\tau)\Delta}) F(\tau) \|_{q} d \tau \\
& \le 
C \int_{0}^{\frac{t}{2}} (1+t- \tau)^{-\frac{n}{2}(1-\frac{1}{q})-1}  \| F(\tau) \|_{1} d \tau
+C \int_{0}^{\frac{t}{2}} e^{-\frac{t-\tau}{2}} \| F(\tau) \|_{q} d \tau \\
& \le C (1+t)^{-\frac{n}{2}(1-\frac{1}{q})-1} 
\int_{0}^{\frac{t}{2}} (1 + \tau)^{-\frac{n}{2}(p-1)} d \tau \\
& +C e^{-\frac{t}{2}} 
\int_{0}^{\frac{t}{2}}  (1 + \tau)^{-\frac{n}{2}(p-1)-\frac{n}{2}(1-\frac{1}{q})}  d \tau \\
& \le C (1+t)^{-\frac{n}{2}(1-\frac{1}{q})-1}, 
\end{split}
\end{equation*}
which is the desired estimate \eqref{eq:6.3}.
Next, we show \eqref{eq:6.4} wit $j=2$.
By \eqref{eq:4.27} and \eqref{eq:6.1}, 
we see that 
\begin{equation*}
\begin{split}
\| A_{2}(t) \|_{q} 
& \le 
\int_{\frac{t}{2}}^{t} \| K_{1}(t-\tau)F(\tau) \|_{q} d \tau 
 \le 
C \int_{\frac{t}{2}}^{t} \| F(\tau) \|_{q} d \tau \\
& \le C \int_{\frac{t}{2}}^{t}
(1+\tau )^{-\frac{n}{2}(1-\frac{1}{q})-\frac{n}{2}(p-1)} 
 d \tau \\
& \le C (1+t)^{-\frac{n}{2}(1-\frac{1}{q})-\frac{n}{2}(p-1)+1}, 
\end{split}
\end{equation*}
which is the desired estimate \eqref{eq:6.4} with $j=2$.

Thirdly, we show \eqref{eq:6.4} with $j=5$.
By the combination of \eqref{eq:6.1} and the direct computation,
we get
\begin{equation*}
\begin{split}
\| A_{5}(t) \|_{q} 
& \le 
\int_{\frac{t}{2}}^{\infty} 
\|F(\tau)\|_{1} d \tau  \| G_{t}\|_{q} \\
& \le 
\int_{\frac{t}{2}}^{\infty} 
(1+\tau)^{-\frac{n}{2}(p-1)} d \tau  \| G_{t}\|_{q}
\le C t^{-\frac{n}{2}(1-\frac{1}{q})-\frac{n}{2}(p-1)+1},
\end{split}
\end{equation*}
which is the desired estimate \eqref{eq:6.4} with $j=5$.

Let us prove \eqref{eq:6.5}.
To begin with, observe that there exists $\theta \in [0,1]$ such that 
\begin{equation*}
G_{t-\tau}(x-y) -G_{t}(x-y)
= 
(-\tau) \partial_{t} G_{t-\theta \tau}(x-y),
\end{equation*}
because of the mean value theorem on $t$. Then, we can apply \eqref{eq:2.6} with $\tilde{k}=0$, $\ell=1$ and $r=1$ to have
\begin{equation*}
\begin{split}
\| A_{3}(t) \|_{q} 
& \le 
\int_{0}^{\frac{t}{2}} \| (e^{(t-\tau)\Delta} -e^{ t\Delta}) F(\tau) \|_{q} d \tau \\
& = 
\int_{0}^{\frac{t}{2}} \tau \| \partial_{t} e^{(t-\theta \tau)\Delta} F(\tau) \|_{q} d \tau \\
& \le C  
\int_{0}^{\frac{t}{2}} \tau (t- \tau)^{-\frac{n}{2}(1-\frac{1}{q})-1}  \| F(\tau) \|_{1} d \tau \\
& \le C t^{-\frac{n}{2}(1-\frac{1}{q})-1} 
\int_{0}^{\frac{t}{2}} \tau (1 + \tau)^{-\frac{n}{2}(p-1)} d \tau \\
& \le 
\begin{cases}
& C t^{-\frac{n}{2}(1-\frac{1}{q})-1} \log(2+t), \quad p \ge 1+\frac{4}{n}, \\
& C t^{-\frac{n}{2}(1-\frac{1}{q})-\frac{n}{2}(p-1)+1}, \quad 1+\frac{2}{n}< p <1+\frac{4}{n},
\end{cases}
\end{split}
\end{equation*}
which implies \eqref{eq:6.5}.\\

Finally, we prove \eqref{eq:6.6}.
To show the estimate for $A_{4}(t)$, 
we first divide the integrand into two parts:
\begin{equation} \label{eq:6.7}
\begin{split}
& \int_{0}^{\frac{t}{2}} \left(e^{t \Delta} F(\tau, x) - \int_{\R^{n}}F(\tau,y) dy \cdot G_{t}(x) \right) d \tau \\
& = \int_{0}^{\frac{t}{2}}  \int_{|y| \le t^{\frac{1}{4}}} +  \int_{0}^{\frac{t}{2}} \int_{|y| \ge t^{\frac{1}{4}}}
(G_{t}(x-y) -G_{t}(x)) F(\tau,y) dy d \tau
=: A_{41}(t) +A_{42}(t).
\end{split}
\end{equation}
In what follows, we estimate $A_{41}(t)$ and $A_{42}(t)$, respectively.
For the estimate of $A_{41}(t)$,
we apply the mean value theorem again on $x$ to have 
\begin{equation*}
G_{t}(x-y) -G_{t}(x)
= 
(-y) \cdot \nabla_{x} G_{t}(x-\tilde{\theta} y)
\end{equation*}
with some $\tilde{\theta} \in [0,1]$, where $\cdot$ denotes the standard Euclid inner product.
Then we arrive at the estimate 
\begin{equation} \label{eq:6.8}
\begin{split}
\| A_{41}(t) \|_{q} 
& \le 
\int_{0}^{\frac{t}{2}} 
\int_{|y| \le t^{\frac{1}{4}}}
\left\| 
G_{t}(x-y) -G_{t}(x)
\right\|_{L^{q}_{x}}
 |F(\tau,y)| dy 
 d \tau \\
& = 
\int_{0}^{\frac{t}{2}} 
\int_{|y| \le t^{\frac{1}{4}}}
\left\| 
(-y) \cdot \nabla_{x} G_{t}(x-\tilde{\theta} y)
\right\|_{L^{q}_{x}}
 |F(\tau,y)| dy
 d \tau \\
& \le C  t^{-\frac{n}{2}(1-\frac{1}{q})-\frac{1}{2} +\frac{1}{4}} 
\int_{0}^{\frac{t}{2}}  \| F(\tau) \|_{1} d \tau \\
& \le C t^{-\frac{n}{2}(1-\frac{1}{q})-\frac{1}{4}} 
\int_{0}^{\frac{t}{2}} (1 + \tau)^{-\frac{n}{2}(p-1)} d \tau 
 \le C t^{-\frac{n}{2}(1-\frac{1}{q})-\frac{1}{4}},
\end{split}
\end{equation}
by direct calculations.
On the other hand, 
for the term $A_{42}(t)$, 
we recall the fact that \\
$\displaystyle{\int_{0}^{\infty} \int_{\R^{n}}} |F(\tau, y)| dy d \tau < \infty$
implies 
\begin{equation*}
\lim_{t \to \infty}
\int_{0}^{\infty} 
\int_{|y| \ge t^{\frac{1}{4}}}
|F(\tau, y)| dy d \tau =0.
\end{equation*}
Thus we see that 
\begin{equation*}
\begin{split}
\|A_{42}(t) \|_{q} 
& \le 
\int_{0}^{\frac{t}{2}} 
\int_{|y| \ge t^{\frac{1}{4}}}
(\left\|  G_{t}(x-y) 
\right\|_{L^{q}_{x}}
+
\left\|  G_{t}(x)\right\|_{L^{q}_{x}}
)
 |F(\tau,y)| dy 
 d \tau \\
& \le C t^{-\frac{n}{2}(1-\frac{1}{q})}
\int_{0}^{\infty} 
\int_{|y| \ge t^{\frac{1}{4}}}
 |F(\tau,y)| dy
 d \tau,
\end{split}
\end{equation*}
so that 
\begin{equation} \label{eq:6.9}
t^{\frac{n}{2}(1-\frac{1}{q})} \| A_{42}(t) \|_{q} \to 0
\end{equation}
as $t \to \infty$.
Therefore, by combining \eqref{eq:6.7}, \eqref{eq:6.8} and \eqref{eq:6.9} one has 
\begin{align*}
\|A_{4}(t) \|_{q}
\le \| A_{41}(t) \|_{q}
+
\| A_{42}(t) \|_{q} =o(t^{-\frac{n}{2}(1-\frac{1}{q})})
\end{align*}
as $t \to \infty$,
which is the desired estimate \eqref{eq:6.6}.
We complete the proof of Lemma \ref{Lem:6.2}.
\end{proof}
%

%
\begin{proof}[Proof of Proposition \ref{Prop:6.1}]
For $1 \le q \le \infty$, 
Lemma \ref{Lem:6.2} immediately yields \eqref{eq:6.4}.
Indeed, from \eqref{eq:6.5} - \eqref{eq:6.8} it follows that 
\begin{equation*}
\begin{split}
& t^{\frac{n}{2}(1-\frac{1}{q})}
\left\| 
\left(
\int_{0}^{t} K_{1}(t-\tau) F(\tau) d \tau 
- \int_{0}^{\infty} \int_{\R^{n}}  
 F(\tau, y) dy d \tau \cdot G_{t}(x) 
\right)
\right\|_{q} \\
& \le 
C
t^{\frac{n}{2}(1-\frac{1}{q})}
\sum_{j=1}^{5}
\| A_{j}(t) \|_{q}
\to 0
\end{split}
\end{equation*}
as $t \to \infty$, which is the desired conclusion.
\end{proof}
Now we are in a position to prove Theorem \ref{Thm:1.2}.
\begin{proof}[Proof of Theorem \ref{Thm:1.2}]
From the proof of Theorem \ref{Thm:1.1}, 
we see that the nonlinear term $f(u)$ satisfies the condition \eqref{eq:6.1}.
Then we can apply Proposition \ref{Prop:6.1} to $F(\tau, y) = f(u(\tau,y))$,
and the proof is now complete.
\end{proof}

\vspace*{5mm}
\noindent
\textbf{Acknowledgments. }
The work of the first author (R. IKEHATA) was supported in part by Grant-in-Aid for Scientific Research (C)15K04958 of JSPS.
The work of the second author (H. TAKEDA) was supported in part by Grant-in-Aid for Young Scientists (B)15K17581 of JSPS.



\begin{thebibliography}{99}
\bibitem{BTW} 
Brenner, P., Thom\'ee, V. and Wahlbin, L., 
\textit{Besov spaces and applications to difference methods 
for initial value problems},\ 
Lecture Notes in Mathematics, Vol. 434. Springer-Verlag,
 Berlin-New York, 1975.








\bibitem{DR}
D'Abbicco, M. and Reissig, M.,
Semilinear structural damped waves, Math. Methods Appl. Sci. 37(11)(2014), 1570-1592.

\bibitem{GGS}
Giga, M., Giga, Y. and Saal, J.,
Nonlinear partial differential equations. Asymptotic behavior of solutions and self-similar solutions. 
Progress in Nonlinear Differential Equations and their Applications, 79. Birkh\"auser Boston, Inc., Boston, MA, 2010.


\bibitem{HKN}
Hayashi, N., Kaikina, E.I. and Naumkin, P. I.,
Damped wave equation with super critical nonlinearities, Diff. Int. Eqns 17 (2004), 637-652.


\bibitem{HO}
Hosono, T. and Ogawa, T., 
Large time behavior and  $L^{p}$-$L^{q}$ estimate of $2$-dimensional nonlinear damped wave equations, J. Diff. Eqns 203 (2004), 82-118.



\bibitem{Ik-4}
Ikehata, R., 
Asymptotic profiles for wave equations with strong damping, J. Diff. Eqns 257 (2014), 2159-2177.


\bibitem{IMN}
Ikehata, R., Miyaoka, Y. and Nakatake, T., 
Decay estimates of solutions for dissipative wave equations in ${\bf R}^N$ with lower power nonlinearities, J.Math.Soc.Japan 56 (2004), 365-373.



\bibitem{IS}
Ikehata, R. and Sawada, A., 
Asymptotic profiles of solutions for wave equations with frictional and viscoelastic damping terms, 
Asymptotic Anal. 98 (2016), 59-77. DOI\,10.3233/ASY-161361 


\bibitem{IT}
Ikehata, R. and Takeda, H., 
Critical exponent for nonlinear wave equations with frictional and viscoelastic damping terms,
arXiv:1604.08265v1 [math.AP] 27 Apr 2016.


\bibitem{ITani}
Ikehata, R. and Tanizawa, K., 
Global existence of solutions for semilinear damped wave equations in ${\bf R}^{N}$ with noncompactly supported initial data, 
Nonlinear Anal. 61 (2005), 1189-1208.


\bibitem{ITY}
Ikehata, R., Todorova, G. and Yordanov, B., 
Wave equations with strong damping in Hilbert spaces, J. Diff. Eqns 254 (2013), 3352-3368.


\bibitem{K}
Karch, G., 
Selfsimilar profiles in large time asymptotics of solutions to damped wave equations, Studia Math. 143 (2000), 175-197.

\bibitem{KU}
Kawakami, T. and Ueda, Y., 
Asymptotic profiles to the solutions for a nonlinear damped wave equation, Diff. Int. Eqns 26 (2013), 781-814.






\bibitem{MN}
Marcati, P. and Nishihara, K., 
The $L^{p}$-$L^{q}$ estimates of solutions to one-dimensional damped wave equations and their application to compressible flow through porous media, J. Diff. Eqns 191 (2003), 
445-469.


\bibitem{Na} Narazaki, T., 
$L^{p}$-$L^{q}$ estimates for damped wave equations and their applications to semilinear problem, J. Math. Soc. Japan 56 (2004), 585-626.


\bibitem{N-2}
Nishihara, K.,
$L^{p}$-$L^{q}$ estimates to the damped wave equation in $3$-dimensional space and their application, Math. Z. 244 (2003), 631-649.

\bibitem{Ponce}
Ponce, G., 
Global existence of small solutions to a class of nonlinear evolution equations, Nonlinear Anal. 9(5) (19), 399-418.




\bibitem{S}
Segal, I., 
Dispersion for non-linear relativistic equations. II,
Ann. Sci. \'Ecole Norm. Sup. {\bf 1} (1968) 459-497.

\bibitem{Shibata}
Shibata, Y., 
On the rate of decay of solutions to linear viscoelastic equation, Math. Meth. Appl. Sci. 23 (2000), 203-226.




\bibitem{T-2}
Takeda, H., 
Higher-order expansion of solutions for a damped wave equation, Asymptotic Anal. 94 (2015), 1-31. DOI: 10.3233/ASY-151295

\bibitem{TY}
Todorova, G. and Yordanov, B., 
Critical exponent for a nonlinear wave equation with damping, J. Diff. Eqns 174 (2001), 464-489.



\bibitem{Z}
Zhang, Qi S.,
A blow-up result for a nonlinear wave equation with damping: the critical case,
C. R. Acad. Sci. Paris S\'er. I Math. {\bf 333} (2001), 109-114. 
\end{thebibliography}
\end{document}